\documentclass[12pt]{article}
\usepackage{amsthm}
\usepackage{amsmath}
\usepackage{enumerate}
\usepackage{amssymb}
\usepackage{latexsym}
\usepackage{amsfonts}
\usepackage{color}
\usepackage{mathrsfs}
\usepackage{epsfig}
\usepackage{graphicx}
\usepackage{soul}
\newtheorem{theorem}{\bf Theorem}[section]
\newtheorem{proposition}[theorem]{\bf Proposition}
\newtheorem{definition}[theorem]{\bf Definition}

\newtheorem{remark}[theorem]{\bf Remark}
\newtheorem{lemma}[theorem]{\bf Lemma}

\newsavebox{\savepar}

\begin{document}
		\title{A parabolic problem involving $p(x)$-Laplacian, a power and a singular nonlinearity}
			\author{ Akasmika Panda\footnote{akasmika444@gmail.com}$^{~,1}$, Debajyoti Choudhuri\footnote{dc.iit12@gmail.com}$^{~,1}$ \&  Kamel Saoudi\footnote{kmsaoudi@iau.edu.sa (Corresponding author)},$^{\,,2}$\\
				\small{$^1$\it Department of Mathematics, National Institute of Technology Rourkela, India}\\
				\small{$^2$\it Basic and Applied Scientifc Research Center, Imam Abdulrahman Bin Faisal University,,}\\ \small{\it P.O. Box 1982, 31441, Dammam, Saudi Arabia}}
	\date{}
	\maketitle
	\begin{abstract}
	\noindent The purpose of this paper is to study nonlinear singular parabolic equations with $p(x)$- Laplacian. Precisely, we consider the following problem and discuss the existence of a non-negative weak solution.
	\begin{align*}
	\frac{\partial u}{\partial t}-\Delta_{p(x)}u&=\lambda u^{q(x)-1} + u^{-\delta(x)}g+ f&&\text{in}~Q_T,\\
	u&= 0&&\text{on}~\Sigma_T,\\
	u(0,\cdot)&=u_0(\cdot)&&\text{in}~\Omega\nonumber.
	\end{align*}
	Here $Q_T=\Omega\times(0,T)$, $\Sigma_T=\partial\Omega\times(0,T)$, $\Omega$ is a bounded domain in $\mathbb{R}^N$ ($N\geq 2$) with Lipschitz continuous boundary $\partial\Omega$, $\lambda\in(0,\infty)$, $f\in L^1(Q_T)$, $g\in L^\infty(\Omega)$, $u_0\in L^r(\Omega)$ with $r\geq 2$,  $\delta:\overline{\Omega}\rightarrow(0,\infty)$ is continuous, and $p,q\in C(\overline{\Omega})$ with $\underset{x\in\overline{\Omega}}{\max}~p(x)<N$,  $q(\cdot)<p^*(\cdot)$.\\
	The article is distinguished into two cases according to the choice of $f$ with different range of
	parameters $p(\cdot)$, $q(\cdot)$.
	\end{abstract}
	\section{Introduction}
	Throughout the article, we will consider a bounded domain $\Omega$ in $\mathbb{R}^N$ ($N\geq 2$) with a Lipschitz continuous boundary $\partial\Omega$. We denote $Q_T=\Omega\times(0,T)$, $\Sigma_T=\partial\Omega\times(0,T)$ for final time $T>0$, and $$C_+(\overline{\Omega})=\{f\in C(\overline{\Omega}): 1<f^-\leq f(x)\leq f^+<\infty,~\forall x\in \overline{\Omega}\}$$ where $f^+=\underset{\overline{\Omega}}{\sup}\{f(x)\}$ and $f^-=\underset{\overline{\Omega}}{\inf}\{f(x)\}$. Further, consider $\delta:\overline{\Omega}\rightarrow(0,\infty)$ to be a continuous function, and $p,q\in C_+(\overline{\Omega})$. \\
	In this article, we study singular problems according to the choice of $f$ and $\lambda$. These type problems are related to different models such as turbulent flow of gas in porous media \cite{Regan}, chemical heterogeneous catalyst kinetics \cite{Bank}, thermo-conductivity \cite{Fulks}, electromagnetic field \cite{Ghergu}, boundary layer phenomena for viscous fluids, signal transmission, non-Newtonian flows, etc \cite{Gatica,Nachman}.
\subsection{The first problem}
	The first part of the article is devoted to the study of the following singular parabolic problem with an $L^1$ datum given by
		\begin{align}\label{main para}
		\frac{\partial u}{\partial t}-\Delta_{p(x)}u&=\lambda u^{q(x)-1} + u^{-\delta(x)}g+ f&&\text{in}~Q_T,\nonumber\\
			u&> 0&&\text{in}~Q_T,\nonumber\\
		u&= 0&&\text{on}~\Sigma_T,\\
		u(0,\cdot)&=u_0(\cdot)&&\text{in}~\Omega,\nonumber
		\end{align}
where $\lambda>0$, the functions $u_0\in L^r(\Omega)$ with $r\geq 2$, $f\in L^1(Q_T)$, $g\in L^\infty(\Omega)$ are positive, and $p,q$ satisfy the following hypotheses:
\begin{itemize}
	\item [(A1)] $2-\frac{1}{N+1}<p^-\leq p^+<N$
	\item[(A2)] $q^+<p^-+\frac{1}{N+1}$
\end{itemize}
To the best of our knowledge, parabolic problems of type $\eqref{main para}$ which involves a power nonlinearity, singularity and an $L^1$ datum with the $p(x)$-Laplace operator are novel and new in the literature.\\
The literature concerning the elliptic counterpart of the problem we here consider is huge. Therefore, we will only refer the readers to the works which have some connection to the current one. The weak theory for purely singular problems has been developed over many years starting from the works by \cite{Boccardo elliptic,Bougherara elliptic,Orsina} with linear operators, and by \cite{Canino,Oliva elliptic} with nonlinear operators. Elliptic problems involving a singular nonlinearity, a measure data or an $L^1$ data have been studied in \cite{Cave,Oliva elliptic,Panda} and the references therein. Since problems of type $\eqref{main para}$ are new in the literature, we find a very less number of articles dealing with its stationary case. Further, we cite \cite{Ghosh} where the authors have settled the multiplicity result for the stationary problem of $\eqref{main para}$ with $p(\cdot)=p$ (a constant). \\
Let us now discuss some of the important parabolic problems which has helped us in the development of this work. Concerning the case $\lambda=0$ and $p(\cdot)=p$ of $\eqref{main para}$, the existence result has been investigated in \cite{Bonis,Bonis 1,Bonis 2,Oliva para} and the bibliography therein. de Bonis \& Giachetti in \cite{Bonis}, assumed the functions $f$ and $g$ to be in some $L^r\left(0,T;L^m(\Omega)\right)$ space with $\frac{1}{r}+\frac{N}{pm}<1$ in order to have bounded solutions. In the same spirit of \cite{Bonis}, Oliva and Petitta in \cite{Oliva para} have shown the existence of a weak solution considering $g$ to be in $L^1(Q_T)$ and by replacing the function $f$ with a bounded Radon measure $\mu$. Parabolic problems as $\eqref{main para}$ with $\lambda=0$, $g=0$, and measure or $L^1$ data have been analyzed by several authors since the papers \cite{Bendahmane,Sabri,Zhang}. The case $p(x)=$ const. in $\eqref{main para}$ is studied by Petitta et al. in \cite{Petitta,Petitta 1}, and Boccardo et al. in  \cite{Boccardo 1989,Boccardo 1997}, always assuming $\lambda=0$, $g=0$ and for measure or $L^1$ data. Most of them worked with renormalized solutions and entropy solutions. It is worth mentioning the result in \cite{Jamea} where nonlinear parabolic problems with variable exponent are considered with Neumann-type boundary conditions.
\subsection{The second problem}
In the second part, the singular parabolic problem we are going to study is the following:
	\begin{align}\label{main para 2}
\frac{\partial u}{\partial t}-\Delta_{p(x)}u&=\lambda u^{q(x)-1} + u^{-\delta(x)}{g}&&\text{in}~Q_T,\nonumber\\
u&> 0&&\text{in}~Q_T,\nonumber\\
u&= 0&&\text{on}~\Sigma_T,\\
u(0,\cdot)&=u_0(\cdot)&&\text{in}~\Omega,\nonumber
\end{align}
where $0<u_0\in L^r(\Omega)$, $\lambda>0$, and $p,q,r$ obey the following restrictions:
\begin{itemize}
	\item[(B1)] $2\leq p^-\leq p^+<N$
	\item[(B2)] $p(x)\leq q(x)<p^*(x)=\frac{Np(x)}{N-p(x)}$ for all $x\in \overline{\Omega}$
	\item[(B3)] $q^+<p^-\left(1+\frac{r}{N}\right)$
	\item[(B4)] $r>\max\{q^+,\delta^++1\}$
\end{itemize}
In the elliptic setting, the literature for the problem as in $\eqref{main para 2}$ with $p(x)$-Laplacian or $p$-Laplacian can be found in \cite{Giacomoni,Haitao,Hirano,Saoudi} and the bibliography therein. More precisely, these seminal papers deal with the existence and multiplicity of the problem both in the subcritical and critical case.\\
For the parabolic case, we refer to the works by Badra et al. in \cite{Badra 1,Badra 2}, Bougherara et al. in \cite{Bougherara,Bougherara 1} and the references therein for model problems as in $\eqref{main para 2}$. More precisely, the paper considered by M. Badra, K. Bal and J. Giacomoni in \cite{Badra 1} is as follows: 
\begin{align}\label{badra}
\frac{\partial u}{\partial t}-\Delta_{p}u&= f(x,u,\nabla u) + u^{-\delta}&&\text{in}~Q_T,\nonumber\\
u&> 0&&\text{in}~Q_T,\nonumber\\
u&= 0&&\text{on}~\Sigma_T,\\
u(0,\cdot)&=u_0(\cdot)&&\text{in}~\Omega,\nonumber
\end{align}
where $p\in(1,\infty)$, $\delta<2+\frac{1}{p-1}$, $u_0\in W_0^{1,p}(\Omega)\cap L^\infty(\Omega)$ satisfies a cone condition, and $f(x,s,\xi)=f(x,s)$ is a Caratheodory function which is bounded from below, locally Lipschitz in the second variable and it satisfies the following subhomogeneous growth condition:
$$0\leq \lim\limits_{s\rightarrow \infty}\frac{f(x,s)}{s^{p-1}}=\alpha_f<\lambda_1(\Omega).$$
Here, $\lambda_1(\Omega)$ is the first eigenvalue of $-\Delta_p$ in $\Omega$ with zero Dirichlet boundary condition. The authors have proved the existence and uniqueness of a weak solution to $\eqref{badra}$. Later, in 2015, Bougherara \& Giacomoni \cite{Bougherara 1} generalized the result in \cite{Badra 1} for any $p>1$, $\delta>0$, and $u_0\in (C_0(\Omega))^+$. Furthermore, Bougherara et al. \cite{Bougherara} studied the problem $\eqref{badra}$ with $\delta>0$, $u_0\in L^r(\Omega)$ and  $$f(x,s,\xi)\leq as^{q-1}+b+c|\xi|^{p-\frac{p}{q}},~\forall x\in\Omega,s\in\mathbb{R}^+,|\xi|\geq M,$$ 
where $a,c,M>0$, $b\geq 0$ , $r\geq q$, $2\leq p\leq q<\max\left\{p^*,p\left(1+\frac{r}{N}\right)\right\}$. Motivated by the former result in \cite{Bougherara}, we  study $\eqref{main para 2}$ for the variable exponent case. 
	\subsection{Plan of the paper}
In the present work, our objective is to construct some auxiliary problems where we replace the singular term by $1/(u+1/n)^{\delta(\cdot)}$ for every $n\in\mathbb{N}$ and find some required a priori estimates. Then, with the help of these estimates, we pass the limit $n\rightarrow\infty$ in the auxiliary problems to obtain a weak solution to $\eqref{main para 2}$ and $\eqref{main para}$.\\
We now describe the plan of the paper. In Section $\ref{section 1}$, we provide the mathematical preliminaries that will be helpful throughout the paper. In addition to that, we define the notion of weak solutions to problems $\eqref{main para}$, $\eqref{main para 2}$, and also state our main results. In Section $\ref{section 2}$, we discuss a general singular parabolic problem with an $L^1$ datum. Section $\ref{section 4}$ is all about proving the existence result for $\eqref{main para}$ with a restriction on the class of $L^1$ functions $f$. Section $\ref{section 5}$ is followed by two subsections to prove the existence result for $\eqref{main para 2}$. In Subsection $\ref{section 5.2}$, we find some a priori estimates on the sequence of the solutions of the approximating problems and in Subsection $\ref{section 5.1}$ we prove Theorem $\ref{main result 2}$. In the Appendix, we present an alternate approximation scheme with the help of a semi-discretization approach in time. 
	\section{Mathematical preliminaries and main results}\label{section 1}
	Consider the domain $\Omega\subset\mathbb{R}^N$ to be bounded, and the function $p\in C_+(\overline{\Omega})$. Then, the Lebesgue space with variable exponent $p(\cdot)$ is defined by
	$$L^{p(\cdot)}(\Omega)=\left\{u:\Omega\rightarrow\mathbb{R}:\int_{\Omega}|u|^{p(x)}<\infty\right\}.$$
	The space $L^{p(\cdot)}(\Omega)$ is a reflexive Banach space endowed with the following norm:
	$$\|u\|_{L^{p(\cdot)}(\Omega)}=\inf\left\{\mu>0:\int_{\Omega}\left|\frac{u(x)}{\mu}\right|^{p(x)}dx<1\right\}.$$
Let us define the modular function as $\rho(u)=\int_{\Omega}|u(x)|^{p(x)}dx$. Then, the relations between the modular function and the norm $\|\cdot\|_{L^{p(\cdot)}(\Omega)}$ are as follows:
\begin{itemize}
	\item $\|u\|_{L^{p(\cdot)}(\Omega)}=\mu\iff \rho(\frac{u}{\mu})=1.$
	\item $\|u\|_{L^{p(\cdot)}(\Omega)}>1\implies\|u\|_{L^{p(\cdot)}(\Omega)}^{p^-}\leq\rho(u)\leq\|u\|_{L^{p(\cdot)}(\Omega)}^{p^+}$. This inequality reverses if  $\|u\|_{L^{p(\cdot)}(\Omega)}<1$.
	\item $\underset{{n\rightarrow\infty}}{\lim}\|u_n-u\|_{L^{p(\cdot)}(\Omega)}=0\iff\underset{{n\rightarrow\infty}}{\lim}\rho(u_n-u)=0.$
\end{itemize}
The dual space of $L^{p(\cdot)}$ is denoted by $L^{p^\prime(\cdot)}$ where $p^\prime(\cdot)=\frac{p(\cdot)}{p(\cdot)-1}$. If $u_1\in L^{p(\cdot)}(\Omega)$ and $u_2\in L^{p^\prime(\cdot)}(\Omega)$, then we have the following H\"{o}lder type inequality:
$$\int_{\Omega}|u_1 u_2|dx\leq \left(\frac{1}{p^-}+\frac{1}{p^{\prime -}}\right) \|u_1\|_{L^{p(\cdot)}(\Omega)}\|u_2\|_{L^{p^\prime(\cdot)}(\Omega)}.$$
Let $q\in C(\overline{\Omega})$ with $q\geq p$, then $L^{q(\cdot)}(\Omega)$ is continuously embedded in $L^{p(\cdot)}(\Omega)$. \\	
The Sobolev space with variable exponent is given by
$$W^{1,p(\cdot)}(\Omega)=\left\{u\in L^{p(\cdot)}(\Omega): |\nabla u|\in L^{p(\cdot)}(\Omega) \right\}.$$
The space $W^{1,p(\cdot)}(\Omega)$ with $p^->1$ is a separable and reflexive Banach space equipped with the following norm:
$$\|u\|_{W^{1,p(\cdot)}(\Omega)}=\|u\|_{L^{p(\cdot)}(\Omega)}+\||\nabla u|\|_{L^{p(\cdot)}(\Omega)}.$$
We also define the subspace $W_0^{1,p(\cdot)}(\Omega)$ as closure of  $C_c^\infty(\Omega)$ in ${W^{1,p(\cdot)}(\Omega)}$. The dual of $W_0^{1,p(\cdot)}(\Omega)$ is denoted by $W_0^{-1,p^\prime(\cdot)}(\Omega)$.	For detailed study on these variable exponent spaces one can refer the work of Fan \& Zhao \cite{Fan}.\\
By the Poincar\'{e} inequality (\cite{Fan}), for every $u\in W_0^{1,p(\cdot)}(\Omega)$ we have
$$\|u\|_{L^{p(\cdot)}(\Omega)}\leq C\||\nabla u|\|_{L^{p(\cdot)}(\Omega)}$$
where $C=C(\Omega,p)>0$, and by the Sobolev embedding theorem, the following embedding:
 $$W^{1,p(\cdot)}(\Omega)\hookrightarrow L^{r(\cdot)}(\Omega)$$ is continuous for any $r\in C(\overline{\Omega})$ with $r(\cdot)\leq p^*(\cdot)=\frac{Np(\cdot)}{N-p(\cdot)}$. Moreover, this embedding is compact for any $r\in C(\overline{\Omega})$ with $\underset{x\in \overline{\Omega}}{\inf}(p^*(x)-r(x))>0$.\\
 Let us consider the extended function $p:\overline{Q}_T= \overline{\Omega}\times [0,T]\rightarrow[1,\infty)$ with $p(t,x)=p(x)$ for every $(t,x)\in \overline{Q}_T$. We now define a generalized Lebesgue space with variable exponent by
 $$L^{p(\cdot)}(Q_T)=\left\{u:Q_T\rightarrow\mathbb{R}:\int_{Q_T}|u|^{p(x)}dx dt<\infty\right\}.$$
The properties of $L^{p(\cdot)}(Q_T)$ are the same as $L^{p(\cdot)}(\Omega)$ when endowed with the norm $\|\cdot\|_{L^{p(\cdot)}(Q_T)}$ given by
 $$\|u\|_{L^{p(\cdot)}(Q_T)}=\inf\left\{\mu>0:\int_{Q_T}\left|\frac{u(x,t)}{\mu}\right|^{p(x)}dx dt<1\right\}.$$
 If $X$ is a Banach space, then the space $L^r(0,T;X)$ with $r\geq 1$ denotes the standard Bochner space and is defined as follows:
 $$L^r(0,T; X)=\left\{u:(0,T)\rightarrow X: \|u(t)\|_{X}\in L^r(0,T)\right\}.$$
Furthermore, $C([0,T];X)$ identifies the space of continuous functions $u:[0,T]\rightarrow X$ such that $\|u\|_{C([0,T];X)}=\underset{t\in[0,T]}{\max}\|u(t)\|_{X}<\infty$. 
\begin{remark}
	We have the following embedding results (Refer \cite{Bendahmane}):
	\begin{enumerate}
		\item Let $p(\cdot)$ and $q(\cdot)$ be two continuous functions with $p(x)\leq q(x)$ for almost every $x\in \Omega$. Then, the embedding from $L^{q^-}(0,T;W_0^{1,q(\cdot)}(\Omega))$ to $L^{p^-}(0,T;W_0^{1,p(\cdot)}(\Omega))$ is continuous.
		\item The following inclusions are continuous and dense: $$L^{p^+}(0,T;L^{p(\cdot)}(\Omega))\overset{d}{\hookrightarrow} L^{p(\cdot)}(Q_T)\overset{d}{\hookrightarrow} L^{p^-}(0,T;L^{p(\cdot)}(\Omega)).$$
	\end{enumerate}
\end{remark}  
\noindent With the consideration of the above results and remarks, we introduce a natural function space with variable exponent as follows:
\begin{equation}\label{V}
V^{p(\cdot)}(Q_T)=\left\{u\in L^{p^-}(0,T,W_0^{1,p(\cdot)}(\Omega)): |\nabla u|\in L^{p(\cdot)}(Q_T) \right\}.
\end{equation}
The space $V^{p(\cdot)}(Q_T)$ is a separable and reflexive Banach space endowed with the norm
$$\|u\|_{V^{p(\cdot)}(Q_T)}=\|\nabla u\|_{L^{p(\cdot)}(Q_T)}.$$
According to Bendahmane et al. in \cite{Bendahmane}, we have the following continuous dense embeddings:
\begin{enumerate}
	\item $L^{p^+}(0,T;W_0^{1,p(\cdot)}(\Omega))\overset{d}{\hookrightarrow} V^{p(\cdot)}(Q_T)\overset{d}{\hookrightarrow} L^{p^-}(0,T;W_0^{1,p(\cdot)}(\Omega))$.
	\item $L^{(p^-)'}(0,T;(W^{1,p(\cdot)}(\Omega))^*)\overset{d}{\hookrightarrow} V^{p(\cdot)}(Q_T)^*\overset{d}{\hookrightarrow} L^{(p^+)'}(0,T;(W^{1,p(\cdot)}(\Omega))^*)$, where $V^{p(\cdot)}(Q_T)^*$ is the dual space of $V^{p(\cdot)}(Q_T)$.
\end{enumerate}
In this paper, we are also going to deal with local Sobolev spaces with variable exponent defined by
\begin{equation}\label{local}
V_{loc}^{p(\cdot)}(Q_T)=\left\{u:Q_T\rightarrow \mathbb{R}: u~\text{and}~|\nabla u|\in L^{p(\cdot)}(K\times(0,T))~\text{for every compact}~K\subset\Omega\right\}.
\end{equation}
Thus, we need to define a general sense of trace known as $M$-boundary trace which is given below.
\begin{definition}
	Let $\{\Omega_m\}$ be a sequence such that $\bar{\Omega}_m\subset\Omega_{m+1}\subset\Omega$. Then, $\{\Omega_m\}$ is said to be an exhaustion of $\Omega$ if $\Omega_m\uparrow\Omega$. If each $\Omega_m$ is of $C^2$ class, then this exhaustion is said to be of class $C^2$. Moreover, we say that an exhaustion $\{\Omega_m\}$ is a uniform $C^2$ exhaustion if $\Omega$ is $C^2$  and  the sequence $\{\Omega_m\}$ is uniformly of class $C^2$.
\end{definition}
\begin{definition}[M-boundary trace, \cite{Giri}]\label{trace}
	Let $u\in W_{loc}^{1,p(\cdot)}(\Omega)$ for $p^->1$. Then, $\nu\in\mathfrak{M} (\partial \Omega)$ is said to be the $M$-boundary trace of $u$ on $\partial\Omega$ if for every $C^2$ exhaustion $\{\Omega_m\}$ and for every $f\in C(\bar{\Omega})$ 
	$$\int_{\partial \Omega_m} u\lfloor_{\partial \Omega_m} f dS\rightarrow \int_{\partial \Omega} f d\nu.$$
	Here, $\mathfrak{M} (\partial \Omega)$ is the space of bounded Borel measures on $\partial\Omega$ with the total variation norm, $u\lfloor_{\partial \Omega_m}$ denotes the Sobolev trace, $dS= dH^{N-1}$ and $H^{N-1}$ denote the $(N-1)$ dimensional Hausdorff measure. The $M$-boundary trace $\nu$ of $u$ is denoted by $tr~u$.\\
	Furthermore, if $u\in W^{1,p(\cdot)}(\Omega)$, then the Sobolev trace of $u$ is the same as $M$-boundary trace of $u$ = $tr~u$.
\end{definition}
\begin{remark}
\begin{enumerate}
\item 
For a motivation behind taking this as the definition of the boundary trace, refer \cite{marcus}; Section $1.3$, p. $14$.
\item Throughout the manuscript, we will use the boundary condition in the sense of the Definition \ref{trace}. 
\end{enumerate}
\end{remark}
\noindent Before providing the notion of the weak solutions to the problems $\eqref{main para}$ and $\eqref{main para 2}$, we define the truncation functions, which will be used
henceforth very often. For a fixed $k>0$, we define the truncation function $T_k$ as $T_k(s)=\max\{-k,\min\{s,k\}\}$ and the level set as $G_k(s)=(|s|-k)^+sign(s)$. For $\gamma>0$ we define 
\begin{equation}\label{truncation}
T_{k,\gamma}(s)=\int_{0}^{s}T_k^\gamma(\tau) d\tau
\end{equation} and 
\begin{equation}\label{second truncation}
V_\gamma(s)=\begin{cases}
1 & s\leq \gamma\\
\frac{2\gamma-s}{\gamma} & k<s\leq 2\gamma\\
0 & s\geq 2\gamma. 
\end{cases}
\end{equation}
\noindent We now introduce the definition of solution we consider for problem $\eqref{main para}$, as well as the main results we prove concerning $\eqref{main para}$.
	\begin{definition}\label{weak main para}
	Assume $\delta^+<1$. Then, a weak solution to the problem $\eqref{main para}$ is a function $u\in L^1\left(0,T;W_0^{1,1}(\Omega)\right)$ such that, $\frac{g}{u^{\delta(\cdot)}}\in L^1\left(0,T;L^1_{loc}(\Omega)\right)$ and
		\begin{equation}\label{weak form main para}
		-\int_{Q_T}u\varphi_t-\int_{\Omega}u_0\varphi(x,0)+\int_{Q_T}|\nabla u|^{p(x)-2}\nabla u\cdot \nabla\varphi=\lambda\int_{Q_T}u^{q(x)-1}\varphi+\int_{Q_T}\frac{g\varphi}{u^{\delta(x)}}+\int_{\Omega}f\varphi
		\end{equation}
		for every $\varphi\in C_c^1(\Omega\times[0,T))$.\\
		Further, assume that $\delta^+\geq 1$. Then, a function $u\in L^1\left(0,T;W_{loc}^{1,1}(\Omega)\right)$ is said to be a weak solution to $\eqref{main para}$ if $\frac{g}{u^{\delta(\cdot)}}\in L^1\left(0,T;L^1_{loc}(\Omega)\right)$, $tr~u(\cdot,t)=0$ in the sense of Definition $\ref{trace}$ for every $t\in (0,T)$ and $u$ satisfies $\eqref{weak form main para}$ for every $\varphi \in C_c^1(\Omega\times[0,T))$.
	\end{definition}
\noindent Before moving towards the existence result of $\eqref{main para}$, we consider the following auxiliary singular problem, which is a type of problem $\eqref{main para}$ with $\lambda=0$:
\begin{align}\label{helping sing}
\frac{\partial v}{\partial t}-\Delta_{p(x)}v&= v^{-\delta(x)}g+ 2f&&\text{in}~Q_T,\nonumber\\
v&> 0&&\text{in}~Q_T,\nonumber\\
v&= 0&&\text{on}~\Sigma_T,\\
v(0,\cdot)&=v_0(\cdot)&&\text{in}~\Omega,\nonumber
\end{align}
where $u_0\leq v_0\in L^r(\Omega)$ with $r\geq 2$.\\
\noindent We now state our first and second main results of the paper in the following theorems.
\begin{theorem}\label{theorem 1}\label{help sing exist}
	Let the assumptions $(A1)$-$(A2)$ be satisfied. Then, for $\delta^+<1$, there exists a non-negative weak solution $v$ to $\eqref{helping sing}$  in $V^{r(\cdot)}(Q_T)$ for every $1\leq r(\cdot)<p(\cdot)-\frac{N}{N+1}$, in the sense of Definition $\ref{weak main para}$, and $T_k(v)\in V^{p(\cdot)}(Q_T)$ for every $k>0$. Similarly, for $\delta^+\geq 1$, problem $\eqref{helping sing}$ admits a non-negative weak solution $v\in V_{loc}^{r(\cdot)}(Q_T)$ for every $1\leq r(\cdot)<p(\cdot)-\frac{N}{N+1}$, in the sense of Definition $\ref{weak main para}$, and $T_k(v)\in V_{loc}^{p(\cdot)}(Q_T)$ for every $k>0$.
\end{theorem}
	\begin{remark}
	The same existence result holds if we replace $f$ by a positive bounded Radon measure $\mu$ in the problem $\eqref{helping sing}$.
\end{remark}
\begin{remark}\label{remark L1}
	\begin{enumerate}
		\item By the assumption $(A1)$, we have $p(\cdot)-\frac{N}{N+1}>1$. This guarantees the existence of functions $r(\cdot)$ with $1\leq r(\cdot)<p(\cdot)-\frac{N}{N+1}$.
		\item By the assumption $(A2)$, we have $q^+-1<p^- -\frac{N}{N+1}\leq p(\cdot)-\frac{N}{N+1}$. Thus, according to Theorem $\ref{help sing exist}$, for $\delta^+< 1$, the solution $v$ to $\eqref{helping sing}$ also belongs to $V^{q^+-1}(Q_T)\cap L^{q(\cdot)-1}(Q_T)$. 
	\end{enumerate}
\end{remark}
	\begin{theorem}\label{main result 1}
		Let $\lambda\leq 1$, $\delta^+<1$, and let the assumptions $(A1)$-$(A2)$ be satisfied. Further, assume that $f\in L^1(Q_T)$ verifies $f\geq v^{q(\cdot)-1}$ a.e. in $Q_T$, where $v$ is a non-trivial weak solution to $\eqref{helping sing}$ as obtained in Theorem $\ref{help sing exist}$. Then $\eqref{main para}$ admits a non-negative weak solution $u$ in $V^{r(\cdot)}(Q_T)$ for every $1\leq r(\cdot)<p(\cdot)-\frac{N}{N+1}$, in the sense of Definition $\ref{weak main para}$. Further, $T_k(u)\in V^{p(\cdot)}(Q_T)$ for every $k>0$.
	\end{theorem}
	\noindent Next we define the notion weak solution to $\eqref{main para 2}$ as follows.
	\begin{definition}\label{weak main para 2}
			Assume $\delta^+<1$. Then, a weak solution to the problem $\eqref{main para 2}$ is a function $u$ such that, $u\in V^{p(\cdot)}(Q_T)\cap L^\infty\left(0,T;L^r(\Omega)\right)\cap L^\infty((\eta,T)\times\Omega)$ for every $\eta\in (0,T)$, $\frac{{g}}{u^{\delta(\cdot)}}\in L^1\left(0,T;L^1_{loc}(\Omega)\right)$ and  for every $\varphi\in C_c^1(\Omega\times[0,T))$,
			\begin{equation}\label{weak form main para 2}
			-\int_{Q_T}u\varphi_t-\int_{\Omega}u_0\varphi(x,0)+\int_{Q_T}|\nabla u|^{p(x)-2}\nabla u\cdot \nabla\varphi=\lambda\int_{Q_T}u^{q(x)-1}\varphi+\int_{Q_T}\frac{{g}\varphi}{u^{\delta(x)}}.
			\end{equation}
			Further, assume that $\delta^+\geq1$. Then, a function $u$ is said to be a weak solution to $\eqref{main para 2}$ if $u\in V_{loc}^{p(\cdot)}(Q_T)\cap L^\infty\left(0,T;L^r(\Omega)\right)\cap L^\infty((\eta,T)\times\Omega)$ for every $\eta\in (0,T)$, $\frac{f}{u^{\delta(\cdot)}}\in L^1\left(0,T;L^1_{loc}(\Omega)\right)$, $tr~u(\cdot,t)=0$ in the sense of Definition $\ref{trace}$ for every $t\in (0,T)$ and $u$ satisfies $\eqref{weak form main para 2}$ for every $\varphi \in C_c^1(\Omega\times[0,T))$.
	\end{definition}
	\noindent The following theorem is the existence result for $\eqref{main para 2}$ and is our third main result.
	\begin{theorem}\label{main result 2}
		Let the assumptions $(B1)-(B4)$ hold. Then, there exists $\overline{T}>0$ such that for any $T<\overline{T}$, the problem $\eqref{main para 2}$ admits a weak solution $u$ in the sense of Definition $\ref{weak main para 2}$.
	\end{theorem}
\noindent We now prove a weak comparison principle which is a very useful tool to establish a comparison between singular parabolic problems with $p(x)$-Laplacian.
	\begin{theorem}[Comparison principle]\label{comparison}
		Let $u_0,v_0\in L^2(\Omega)$. Suppose $u,v\in V^{p(\cdot)}(Q_T)$ such that $\frac
		{\partial u}{\partial t}, \frac
		{\partial v}{\partial t}\in V^{p(\cdot)}(Q_T)^*$, and  $$\frac{\partial u}{\partial t}-\Delta_{p(x)}u-\frac{1}{(u+c)^{\delta(x)}}\leq \frac{\partial v}{\partial t}-\Delta_{p(x)}v-\frac{1}{(v+c)^{\delta(x)}} ~\text{weakly in}~ Q_T, ~c>0,$$ $u(0,\cdot)=u_0(\cdot)\leq v(0,\cdot)=v_0(\cdot)$ in $\Omega$,  $u=v=0$ on $\Sigma_T$. Then, $u\leq v$ a.e. in $Q_T$.
		\begin{proof}
		Since $\frac{\partial u}{\partial t}-\Delta_{p(x)}u-\frac{1}{(u+c)^{\delta(x)}}\leq \frac{\partial v}{\partial t}-\Delta_{p(x)}v-\frac{1}{(v+c)^{\delta(x)}}$ weakly in $Q_T$ with $u=v=0$ on $\Sigma_T$, we have
		$$\int_{Q_T}u_t\varphi+\int_{Q_T}|\nabla u|^{p(x)-2}\nabla u\cdot \nabla\varphi-\int_{Q_T}\frac{\varphi}{(u+c)^{\delta(x)}}\leq \int_{Q_T}v_t\varphi+\int_{Q_T}|\nabla v|^{p(x)-2}\nabla v\cdot \nabla\varphi-\int_{Q_T}\frac{\varphi}{(v+c)^{\delta(x)}}$$ for every $\varphi\in V^{p(\cdot)}(Q_T)$, $\varphi\geq 0$. We rigorously choose $\varphi=(u-v)^+\chi_{\{0,t\}}$ for any $t\in(0,T]$ and we get 
		$$\int_{Q_t}(u-v)_t(u-v)^++\int_{Q_t}\left(|\nabla u|^{p(x)-2}\nabla u-|\nabla v|^{p(x)-2}\nabla v\right)\cdot \nabla(u-v)^+ \leq 0.$$
	Moreover, from the nonnegativity of the second term present in the above inequality, we have
	\begin{align}
	\frac{1}{2}\int_{Q_t}\frac{d}{dt}[(u-v)^+]^2\leq 0.\nonumber
	\end{align}
	Since $u_0\leq v_0$ in $L^2(\Omega)$, we get $(u-v)^+=0$ a.e. in $Q_T$. Thus, $u\leq v$ a.e. in $Q_T$.
		\end{proof}
	\end{theorem}
	\begin{remark}
		We will denote several constants by $C$ which can only depend on $\Omega,~N$ and independent of the indices of the sequences. The value of $C$ can be different from line to line and sometimes, on the same line.
	\end{remark}
	\noindent The general form of an approximating problem to problems of type $\eqref{main para}$ is as follows:
	 \begin{align}\label{approx}
	 \frac{\partial w_n}{\partial t}-\Delta_{p(x)}w_n&=\lambda h_n(w_n)+{(w_n+1/n)^{-\delta(x)}}g+ \beta f_n&&\text{in}~Q_T,\nonumber\\
	 w_n&> 0&&\text{in}~Q_T,\nonumber\\
	 w_n&= 0&&\text{on}~\Sigma_T,\\
	 w_n(0,\cdot)&=w_{0,n}(\cdot)&&\text{in}~\Omega.\nonumber
	 \end{align}
	 	Here, $\lambda,\beta\geq 0$, $h_n(w_n)=T_n\left(w_n^{q(\cdot)-1}\right)$, $w_{0,n}=T_n(w_0)$, $w_0\in L^r(\Omega)$ with $r\geq 2$, and the sequence $\{f_n\}\subset L^\infty(Q_T)$ is such that $f_n\rightarrow f$ strongly in $L^1(Q_T)$. \\
	 	For a fixed $n\in\mathbb{N}$, the right hand side of $\eqref{approx}$ is $L^\infty$ bounded. This allows us to use the standard methods (for example, Schauder's theorem \cite{Oliva para}, variation methods \cite{Zhang} etc.) to obtain the following existence result. An alternate method is discussed in the Appendix with the help of a semi-discretization approach in time. 
		\begin{lemma}\label{exist approx sing}
			For a fixed $n\in \mathbb{N}$, the problem $\eqref{approx}$ admits at least one non-trivial weak solution  $w_n\in V^{p(\cdot)}(Q_T)\cap L^\infty(Q_T)\cap C(0,T;L^2(\Omega))$ with $\frac{\partial w_n}{\partial t}\in V^{p(\cdot)}(Q_T)^*$.
		\end{lemma}
	\section{Proof of Theorem $\ref{help sing exist}$}\label{section 2}
\noindent In this section, we establish the existence result for $\eqref{helping sing}$ by proving Theorem $\ref{help sing exist}$. We follow the method of approximation. For this purpose, consider $v_n$ to be a weak solution to the approximating problem $\eqref{approx}$ with $\lambda=0$, $\beta=2$ and $w_0=v_0$, i.e. $v_n$ satisfies:
 \begin{align}\label{approx helping sing}
 \frac{\partial v_n}{\partial t}-\Delta_{p(x)}v_n&={(v_n+1/n)^{-\delta(x)}}g+ 2 f_n&&\text{in}~Q_T,\nonumber\\
 v_n&> 0&&\text{in}~Q_T,\nonumber\\
 v_n&= 0&&\text{on}~\Sigma_T,\\
 v_n(0,\cdot)&=v_{0,n}(\cdot)&&\text{in}~\Omega.\nonumber
 \end{align}
The corresponding weak formulation is given by 
 \begin{equation}\label{approx weak form helping}
 -\int_{Q_T}v_n\varphi_t-\int_{\Omega}v_{0,n}\varphi(x,0)+\int_{Q_T}|\nabla v_n|^{p(x)-2}\nabla v_n\cdot \nabla\varphi=2\int_{Q_T}f_n\varphi+\int_{Q_T}\frac{g\varphi}{(v_n+1/n)^{\delta(x)}}
 \end{equation}
 \noindent for every $\varphi\in C_c^1(\Omega\times[0,T))$.\\
	\noindent We now prove some required a priori estimates to pass the limit $n\rightarrow\infty$ in $\eqref{approx weak form helping}$ to obtain a weak solution to $\eqref{helping sing}$. Due to the presence of the power $u^{\delta(x)}$ which induces a singular behavior near 0, we divide the proofs among the cases $\delta^+<1$ and  $\delta^+\geq 1$. We deliberately consider $\delta^+\geq1$ as the strongly singular case due to the discontinuity of the enegy functional corresponding to $\eqref{helping sing}$ near 0. If $\delta^+<1$, we find a global solution to $\eqref{helping sing}$; otherwise, we find a local solution.  
	\begin{lemma}\label{infinity bound}
	The sequence $\{v_n\}$ is uniformly bounded in $L^\infty\left(0,T;L^1(\Omega)\right)$.
		\begin{proof}
			Fix a $t\in(0,T]$ and denote $Q_t=\Omega\times (0,t)$. Let us multiply $\eqref{approx helping sing}$ by $T_1^\gamma(v_n)$, for $\gamma=\max\{1,\delta^+\}$. Then, integrating  over $Q_t$, we obtain that
			\begin{align}\label{1}
			\int_{Q_t}(v_n)_t T_1^\gamma(v_n)+\gamma\int_{Q_t}|\nabla T_1(v_n)|^{p(x)}T_1^{\gamma-1}(v_n)=2\int_{Q_t}f_nT_1^\gamma(v_n)+\int_{Q_t}\frac{g T_1^\gamma(v_n)}{(v_n+1/n)^{\delta(x)}}.
			\end{align}
			From the above equation $\eqref{1}$, we get	
			\begin{align}
			\int_{Q_t}(T_{1,\gamma}(v_n))_t&\leq 2 \int_{Q_t}f+\int_{Q_t\cap\{v_n\leq 1\}}v_n^{\gamma-\delta(x)}g+\int_{Q_t\cap\{v_n> 1\}}v_n^{-\delta(x)}g\nonumber\\&\leq C+C\|g\|_{L^\infty(\Omega)},
			\end{align}	
			where the function $T_{1,\gamma}(\cdot)$ is defined in $\eqref{truncation}$. By the definition of $T_{1,\gamma}(\cdot)$ we obtain $T_{1,\gamma}(s)\geq s-1$. Hence,
			\begin{align}
			\int_{\Omega}v_n(x,t)&\leq C+C\|g\|_{L^\infty(\Omega)}+|\Omega|+\int_\Omega T_{1,\gamma}(v_0).\nonumber
			\end{align}
			This implies
			\begin{equation}\label{infinity bounded}
			\|v_n\|_{L^\infty\left(0,T;L^1(\Omega)\right)}\leq C.
			\end{equation}
		\end{proof}
	\end{lemma}
	\begin{lemma}\label{lemma1 para}
		Let $\delta^+<1$. Then, $\{v_n\}$ is bounded in $V^{r(\cdot)}(Q_T)$ for every $1\leq r(x)<p(x)-\frac{N}{N+1}$ for all $x\in\overline{\Omega}$. Moreover, $\{T_k(v_n)\}$ is bounded in $V^{p(\cdot)}(Q_T)$ for any $k>0$.
		\begin{proof}
			Let us multiply $T_k(v_n)$ in $\eqref{approx helping sing}$ and integrate over $Q_T$ to obtain
			\begin{align}
			\int_{Q_T}(v_n)_t T_k(v_n)+\int_{Q_T}|\nabla T_k(v_n)|^{p(x)}&=2\int_{Q_T}f_nT_k(v_n)+\int_{Q_T}\frac{g_nT_k(v_n)}{(v_n+1/n)^{\delta(x)}}\nonumber\\&\leq C_2 k.\nonumber
			\end{align}
		This gives 
			\begin{equation}\label{consider 2}
			\int_{Q_T}|\nabla T_k(v_n)|^{p(x)}\leq C_3k,
			\end{equation}
		and hence
			\begin{equation}\label{Ck}
			\| T_k(v_n)\|_{V^{p(\cdot)}(Q_T)}\leq Ck.
			\end{equation}
By proceeding similarly with the test function $T_1(G_k(v_n))$ we get
\begin{equation}\label{k<k+1}
\int_{Q_T\cap\{k<v_n<k+1\}}|\nabla v_n|^{p(x)}\leq C_4.
\end{equation}
With the consideration of Lemma $\ref{infinity bound}$, $\eqref{k<k+1}$, and Lemma 2.1 of \cite{Bendahmane}, we conclude that $\{v_n\}$ is bounded in $V^{r(\cdot)}(Q_T)$ for every $1\leq r(x)<p(x)-\frac{N}{N+1}$ for all $x\in\overline{\Omega}$.
		\end{proof}
	\end{lemma}
		\begin{lemma}\label{lemma2 para}
			Let $\delta^+\geq 1$. Then, $\{v_n\}$ is bounded in $V_{loc}^{r(\cdot)}(Q_T)$ for every $1\leq r(x)<p(x)-\frac{N}{N+1}$ for all $x\in\overline{\Omega}$. Moreover, $\{T_k(v_n)\}$ is bounded in $V_{loc}^{p(\cdot)}(Q_T)$, and $\left\{T_k^{\frac{p^-+\delta^+-1}{p^-}}(v_n)\right\}$ is bounded in $L^{p^-}(0,T;W_{0}^{1,p^-}(\Omega))$ for any $k>0$.
			\begin{proof}
				The proof follows the method used in Lemma $\ref{lemma1 para}$. Let $\varphi\in C_c^1(\Omega)$ be a non-negative function, and let $k>0$ be a fixed constant. We multiply by $\left(T_1(G_k(v_n))-1\right)\varphi^{p^+}$ in $\eqref{approx helping sing}$, and then integrate by parts on $Q_T$ to get
				\begin{align}\label{para2}
					\int_{Q_T}(v_n)_t (T_1(G_k(v_n))-1)\varphi^{p^+}&+\int_{Q_T}\left|\nabla T_1(G_k(v_n))\right|^{p(x)} \varphi^{p^+}\nonumber\\&+p^+\int_{Q_T}|\nabla v_n|^{p(x)-2}\nabla v_n\cdot \nabla\varphi
					\varphi^{p^+-1}(T_1(G_k(v_n))-1)\leq 0.
				\end{align}
				Using Leibniz integral formula, Lemma $\ref{infinity bound}$ and the non-negativity of $\varphi$, the first term in the left hand ride of $\eqref{para2}$ is estimated as follows:
				\begin{align}
					\left|\int_{Q_T}(v_n)_t (T_1(G_k(v_n))-1)\varphi^{p^+}\right|&\leq\left|\int_\Omega\left(\int_{0}^{v_{0,n}}T_{1}(G_k(s))ds-v_{0,n}\right)\varphi^{p^+}(x)\right|\nonumber\\&~~~+\left|\int_{\Omega}\left(\int_{0}^{v_{n}(T)}T_{1}(G_k(s))ds-v_{n}(T)\right)\varphi^{p^+}(x)\right|\nonumber\\&\leq C.\nonumber
				\end{align}
				Now, with the help of Young's inequality and thanks to $\eqref{para2}$, we have	that
				\begin{align}
					\int_{Q_T}|\nabla T_1(G_k(v_n))|^{p(x)}\varphi^{p^+}&\leq C+\tilde{C}\epsilon p^+ \int_{Q_T}|\nabla T_1(G_k(v_n))|^{p(x)}\varphi^{p^+} +C_\epsilon p^+\int_{Q_T}|\nabla\varphi|^{p(x)}.\nonumber
				\end{align}
				This implies
				\begin{equation}
					\int_{Q_T}|\nabla T_1(G_k(v_n))|^{p(x)}\varphi^{p^+}\leq C,\nonumber
				\end{equation} where $C$ is independent of $n$, and then
				\begin{equation}\label{loc k<k+1}
					\int_{Q_T\cap\{k<v_n<k+1\}}|\nabla v_n|^{p(x)}\varphi^{p^+}\leq C.
				\end{equation}
				 By Lemma $\ref{infinity bound}$, $\eqref{loc k<k+1}$, and Lemma 2.1 of \cite{Bendahmane}, we conclude that $\{v_n\}$ is bounded in $V_{loc}^{r(\cdot)}(Q_T)$ for every $1\leq r(x)<p(x)-\frac{N}{N+1}$ for all $x\in\overline{\Omega}$. Repeating the above calculations with the test function $(T_k(v_n)-k)\varphi^{p^+}$ for any $k>0$ we obatin
				 \begin{equation}\label{consider 1}
				 	\int_{Q_T}|\nabla T_k(v_n)|^{p(x)}\varphi^{p^+}\leq Ck.
				 \end{equation}
				Let us consider $T_k^{\delta^+}(v_n)$ as the test function in $\eqref{approx helping sing}$ and thus we have
				\begin{align}
					\int_{Q_T}(v_n)_t T_k^{\delta^+}(v_n)+\delta^+\int_{Q_T}|\nabla T_k(v_n)|^{p(x)}T_k^{\delta^+-1}(v_n)&=2\int_{Q_T}f_nT_k^{\delta^+}(v_n)+\int_{Q_T}\frac{g_nT_k^{\delta^+}(v_n)}{(v_n+1/n)^{\delta(x)}}\nonumber\\&\leq C_4 k^{\delta^+}.\nonumber
				\end{align}
				By further simplification we get 
				$$\int_{Q_T}(v_n)_tT_k^{\delta^+}(v_n)+\delta^+\int_{Q_T}|\nabla T_k(v_n)|^{p^-}T_k^{\delta^+-1}(v_n)-\delta^+\int_{Q_T}T_k^{\delta^+-1}(v_n)\leq C_4 k^{\delta^+},$$
				and thus
				\begin{equation}
					\int_{Q_T}\left|\nabla T_k^{\frac{p^-+\delta^+-1}{p^-}}(v_n)\right|^{p^-}\leq C_5.\nonumber
				\end{equation}
				Hence, the sequence $\left\{T_k^{\frac{p^-+\delta^+-1}{p^-}}(v_n)\right\}$ is bounded in $L^{p^-}(0,T;W_{0}^{1,p^-}(\Omega))$ for every $k>0$.	
			\end{proof}
		\end{lemma}
\noindent The next theorem is the existence result for $\eqref{helping sing}$.
	\begin{proof}[Proof of Theorem $\ref{help sing exist}$]
	The proof follows the lines used in \cite{Oliva para}. According to Lemma $\ref{lemma1 para}$ and Lemma $\ref{lemma2 para}$, the sequence $\{v_n\}$ is bounded in $V^{r(\cdot)}(Q_T)$ if $\delta^+<1$, and is bounded in $V_{loc}^{r(\cdot)}(Q_T)$ if $\delta^+\geq 1$, for every $1\leq r(\cdot)<p(\cdot)-\frac{N}{N+1}$. Therefore, for $\delta^+<1$, there exists a function $v\in V^{r(\cdot)}(Q_T)$ such that, up to a subsequence, $v_n$ converges to $v$ a.e. in $Q_T$, weakly in $V^{r(\cdot)}(Q_T)$. Similarly, for the case $\delta^+\geq 1$, there exists $v\in V_{loc}^{r(\cdot)}(Q_T)$ such that, up to a subsequence, $v_n$ converges to $v$ a.e. in $Q_T$ and weakly in $V_{loc}^{r(\cdot)}(Q_T)$. Since $\{T_k(v_n)\}$ is bounded in $V_{loc}^{p(\cdot)}(Q_T)$ for any $k>0$ it follows that
	\begin{equation}\label{truncation bounded}
	T_k(v_n)\rightarrow T_k(v) ~\text{weakly in}~ V_{loc}^{p(\cdot)}(Q_T). 
	\end{equation}
	We now need to pass the limit $n\rightarrow \infty$ in the weak formulation $\eqref{approx weak form helping}$.\\
	\textbf{Step 1:} Thanks to Lemma $\ref{infinity bound}$, Lemma $\ref{lemma1 para}$ and Lemma $\ref{lemma2 para}$, considering $0<\varphi\in C_c^1(\Omega)$ as a test function in $\eqref{approx helping sing}$, we have that 
	\begin{align}\label{loc bounded}
	2\int_{Q_T}f_n\varphi+\int_{Q_T}\frac{g \varphi}{(v_n+1/n)^{\delta(x)}}&\leq	C\int_{0}^{T}\int_{supp(\phi)}|\nabla v_n|^{p(x)-1}\nonumber\\&\leq C.
	\end{align}
	Clearly, by $\eqref{loc bounded}$, the right hand side of $\eqref{approx helping sing}$ is bounded in $L^1\left(0,T;L^1_{loc}(\Omega)\right)$. Thus, $\left\{\frac{\partial (v_n\varphi)}{\partial t}\right\}$ is a bounded sequence in $L^{s^-}(0,T;W^{-1,s(\cdot)}(\Omega))+L^1(Q_T)$ with $s(\cdot)=\frac{r(\cdot)}{p(\cdot)-1}$ for any $\varphi\in C_c^1(\Omega)$, $\varphi\geq 0$. This allows us to apply Corollary 4 of \cite{Simon} to guarantee that $v_n$ strongly converges to $v$ in $L^1\left(0,T;L_{loc}^1(\Omega)\right)$.\\
	\textbf{Step 2:} Since $(v_n+1/n)^{-\delta(\cdot)}g$ is bounded in $L^1\left(0,T;L^1_{loc}(\Omega)\right)$ by $\eqref{loc bounded}$, using Fatou's Lemma we observe that $v^{-\delta(\cdot)}g\in L^1\left(0,T;L^1_{loc}(\Omega)\right)$. Thus, by following the work of Oliva \& Petitta [Theorem 2.3, \cite{Oliva para}], it can be proved that for any $\gamma>0$
	\begin{equation}\label{singular term}
	\lim\limits_{\gamma\rightarrow 0}\lim\limits_{n\rightarrow\infty}\int_{Q_T\cap\{v_n\leq \gamma\}}\frac{g}{(v_n+1/n)^{\delta(x)}}\varphi=0.
	\end{equation}
To prove $\eqref{singular term}$, let us consider $V_\gamma(v_n)\varphi$ ($V_\gamma$ is defined in $\eqref{second truncation}$) as a test function in $\eqref{approx helping sing}$. By neglecting a negative term we obtain
\begin{align}
\int_{Q_T\cap\{v_n\leq \gamma\}}\frac{g}{(v_n+1/n)^{\delta(x)}}\varphi&\leq -\int_{Q_T}\left(\varphi_t\int_0^{v_n}V_\gamma(s)ds\right)+\int_{Q_T}|\nabla v_n|^{p(x)-1}|\nabla \varphi|V_\gamma(v_n).\nonumber
\end{align}
From $\eqref{consider 2}$ and $\eqref{consider 1}$ we deduce
$$\lim\limits_{n\rightarrow\infty}\int_{Q_T\cap\{v_n\leq \gamma\}}\frac{g}{(v_n+1/n)^{\delta(x)}}\varphi\leq C\max\{\gamma^{\frac{1}{(p')^-}},\gamma^{\frac{1}{(p')^+}}\}.$$
	This implies $\eqref{singular term}$, and
	\begin{equation}\label{rev_qn}
	\lim\limits_{\gamma\rightarrow 0}\lim\limits_{n\rightarrow\infty}\int_{Q_T\cap\{v_n> \gamma\}}\frac{g}{(v_n+1/n)^{\delta(x)}}\varphi=\int_{Q_T\cap\{v>0\}}\frac{g}{v^{\delta(x)}}\varphi=\int_{Q_T}\frac{g}{v^{\delta(x)}}\varphi,
	\end{equation}
	and $$\{(x,t)\in Q_T:v(x,t)=0\}\subset\{(x,t)\in Q_T:g(x,t)=0\}$$
	except for a set of zero measure. The equation \eqref{rev_qn} is obtained by observing that (upto a subsequence) $g(x)\rightarrow g(x)$, $v_n(x)\rightarrow v(x)$ for a.e. $x\in Q_T$. Thus, we have $$\lim\limits_{n\rightarrow\infty}\int_{Q_T}\frac{g}{(v_n+1/n)^{\delta(x)}}\varphi=\int_{Q_T}\frac{g}{v^{\delta(x)}}\varphi,~\forall \varphi\in C_c^1(\Omega\times[0,T)).$$
\textbf{Step 3:} The next step is to show the almost everywhere convergence of the gradients. Let us consider a compact set $\omega\subset Q_T$, and $\phi\in C_c^1(Q_T)$ with $0\leq \phi\leq 1$ in $Q_T$ and $\phi=1$ on $\omega$. By multiplying $T_k(v_n-v_m)\phi^{p^+}$, $k>0,~n,m\in\mathbb{N}$, in $\eqref{approx helping sing}$ relative to both $v_n$, $v_m$ and then integrating by parts over $Q_T$, we have  
	\begin{align}\label{a.e. first part}
	-\int_{Q_T}(\phi^{p^+})_t & T_{k,1}(v_n-v_m)+\int_{Q_T}\left[|\nabla v_n|^{p(x)-2}\nabla v_n-|\nabla v_m|^{p(x)-2}\nabla v_m\right]\cdot \nabla T_k(v_n-v_m)\phi^{p^+}\nonumber\\&~~~~+p^+\int_{Q_T}\left[|\nabla v_n|^{p(x)-2}\nabla v_n-|\nabla v_m|^{p(x)-2}\nabla v_m\right]\cdot \nabla\phi \phi^{p^+-1} T_k(v_n-v_m)\nonumber\\&=\int_{Q_T} 2(f_n-f_m)T_k(v_n-v_m)\phi^{p^+} +\int_{Q_T}\frac{gT_k(v_n-v_m)\phi^{p^+}}{(v_n+1/n)^{\delta(x)}}-\int_{Q_T}\frac{gT_k(v_n-v_m)\phi^{p^+}}{(v_m+1/m)^{\delta(x)}}\nonumber\\&\leq kC\|f\|_{L^1(Q_T)}+k\int_{Q_T\cap\{v_n\leq \gamma\}}\frac{g\phi^{p^+}}{(v_n+1/n)^{\delta(x)}}+k\int_{Q_T\cap\{v_m\leq \gamma\}}\frac{g\phi^{p^+}}{(v_m+1/m)^{\delta(x)}}\nonumber\\&~~~+\gamma^{-\delta^+}\|g\|_{L^\infty(\Omega)}\left(\int_{Q_T\cap\{v_n>\gamma\}}|T_k(v_n-v_m)|\phi^{p^+}+\int_{Q_T\cap\{v_m>\gamma\}}|T_k(v_n-v_m)|\phi^{p^+}\right),
	\end{align}
where $\gamma\in (0,1)$ and $T_{k,1}$ is the primitive of $T_k$ defined in $\eqref{truncation}$. Since $\{v_n\}$ is uniformly bounded in $V_{loc}^{r(\cdot)}(Q_T)$ for any $r(\cdot)<p(\cdot)-\frac{N}{N+1}$, we obtain
\begin{align}\label{a.e. fourth}
\int_{Q_T}\left|\left[|\nabla v_n|^{p(x)-2}\nabla v_n-|\nabla v_m|^{p(x)-2}\nabla v_m\right]\right|&\cdot  |\nabla\phi| \phi^{p^+-1}  |T_k(v_n-v_m)|\nonumber\\&\leq Ck\left(\int_{Q_T}\left[|\nabla v_n|^{p(x)-1}+|\nabla v_m|^{p(x)-1}\right]\phi^{p^+}\right)\nonumber\\&\leq Ck.
\end{align}
	We use the notation $w(n,m,\gamma)$ for all quanities such that 
	\begin{equation}\label{w}
	\lim\limits_{\gamma\rightarrow0}\lim\limits_{m\rightarrow\infty}\lim\limits_{n\rightarrow\infty}w(n,m,\gamma)=0.
	\end{equation} 
Using $\eqref{singular term}$, $\eqref{a.e. fourth}$, and the facts that $T_k(v_n-v_m)\rightarrow 0$ weakly in $V_{loc}^{p(\cdot)}(Q_T)$, $v_n-v_m\rightarrow 0$ strongly in $L^1(0,T;L^1_{loc}(\Omega))$, we obatin the following.
\begin{equation}
\int_{\omega}\left[|\nabla v_n|^{p(x)-2}\nabla v_n-|\nabla v_m|^{p(x)-2}\nabla v_m\right]\cdot \nabla T_k(v_n-v_m)\leq Ck+ w(n,m,\gamma).\nonumber
\end{equation}
Using simple calculations, for any fixed $0<\theta<1$, we have
\begin{align}\label{a.e. fifth}
\lim\limits_{\gamma\rightarrow 0}&\lim\limits_{n,m\rightarrow\infty}
\int_\omega|\nabla v_n- \nabla v_m|^{\theta p(x)}\nonumber\\&\leq \lim\limits_{\gamma\rightarrow 0}\lim\limits_{n,m\rightarrow\infty}\left(\int_{\omega\cap\{|v_n-v_m|\leq k\}}\left[|\nabla v_n|^{p(x)-2}\nabla v_n-|\nabla v_m|^{p(x)-2}\nabla v_m\right]\cdot(\nabla v_n-\nabla v_m)\right)^\theta|\omega|^{1-\theta}\nonumber\\&\leq (Ck)^{\theta}|\omega|^{1-\theta},
\end{align}
when $p(\cdot)\geq2$. For the case $p(\cdot)<2$, using H\"{o}lder's inequality we obtain $\eqref{a.e. fifth}$.  Letting $k\rightarrow 0$ in $\eqref{a.e. fifth}$ implies that $\nabla v_n$ is cauchy in $\left(L^{\theta{{p}}(\cdot)}(\omega)\right)^N$, and thus $\nabla v_n$ converges strongly to $\nabla v$  in $\left(L_{loc}^{{\tilde{p}}(\cdot)}(Q_T)\right)^N$ for any $\tilde{p}(\cdot)\leq p(\cdot)$. Then, in a standard way it follows that	$\nabla v_n\rightarrow \nabla v~\text{a.e in}~Q_T$.\\
		\textbf{Step 4:} The weak formulation of $\eqref{approx helping sing}$ is given by
		\begin{equation}\label{k}
		 -\int_{Q_T}v_n\varphi_t-\int_{\Omega}v_{0,n}\varphi(x,0)+\int_{Q_T}|\nabla v_n|^{p(x)-2}\nabla v_n\cdot \nabla\varphi=2\int_{Q_T}f_n\varphi+\int_{Q_T}\frac{g\varphi}{(v_n+1/n)^{\delta(x)}}
		\end{equation}
		for every $\varphi\in C_c^1(\Omega\times[0,T))$. By the above arguments we can pass the limit $n\rightarrow\infty$ in the left hand side terms of $\eqref{k}$.\\
		Therefore, after passing the limit, we conclude that $v$ is a weak solution to $\eqref{helping sing}$ in the sense of Definition $\ref{weak main para}$ and satisfies the following equation.
			\begin{equation}
			-\int_{Q_T}v\varphi_t-\int_{\Omega}v_{0}\varphi(x,0)+\int_{Q_T}|\nabla v|^{p(x)-2}\nabla v\cdot \nabla\varphi=2\int_{Q_T}f\varphi+\int_{Q_T}\frac{g\varphi}{v^{\delta(x)}}.
			\end{equation}
\textbf{Step 5:} The final step is to show that $tr~v(\cdot, t)=0$ in the sense of Definition $\ref{trace}$ for almost every $t\in (0,t)$.\\
According to Lemma $\ref{lemma2 para}$, for $\delta^+\geq1$, the sequence $\left\{T_k^{\frac{p^-+\delta^+-1}{p^-}}(v_n)\right\}$ is uniformly bounded in $L^{p^-}(0,T;W_{0}^{1,p^-}(\Omega))$. Thus, $T_k^{\frac{p^-+\delta^+-1}{p^-}}(v(\cdot,t))\in W_{0}^{1,p^-}(\Omega)$ for almost every $t\in (0,T)$ and for every  $k>0$. Let us consider a $C^2$ exhaustion $\{\Omega_m\}$ of $\Omega$, and let $f\in C(\overline{\Omega})$. Then, we establish
\begin{align}\label{trace 1}
\int_{\Omega_m}T_k(v)f&\leq C\left(\int_{\Omega_m}[T_k(v)f]^{\frac{p^-+\delta^+-1}{p^-}}\right)^{\frac{p^-}{p^-+\delta^+-1}}\left|\Omega_m\right|^{\frac{\delta^+-1}{\delta^+-1+p^-}},~C>0,
\end{align}
and, by Definition $\ref{trace}$,  
\begin{equation}\label{trace 2}
\lim\limits_{m\rightarrow \infty} \int_{\partial\Omega_m}T_k^{\frac{p^-+\delta^+-1}{p^-}}(v)\lfloor_{\partial \Omega_m}f^{\frac{p^-+\delta^+-1}{p^-}}~ dS=0,~\forall f\in C(\overline{\Omega}).
\end{equation}
 By $\eqref{trace 1}$ and $\eqref{trace 2}$, it is now easy to prove that $tr~v(\cdot,t)=0$ for almost every $t\in (0,T)$ in the sense of Definition $\ref{trace}$. If $\delta^+<1$, then $T_k(v(\cdot,t))\in W_0^{1,p(\cdot)}(\Omega)$ for almost every $t\in (0,T)$, and thus the Sobolev trace of $v$ is same as $tr~v=0$.
				\end{proof}
	\section{Proof of Theorem $\ref{main result 1}$}\label{section 4}
	In this section, we establish the existence of weak solution to $\eqref{main para}$, i.e. we prove Theorem $\ref{main result 1}$, through approximation. We follow the approach used in Section $\ref{section 2}$. The preliminary step is the construction of the following approximation scheme. For a fixed $n\in \mathbb{N}$, consider a sequence $\{u_{n,j}\}_j$ that verifies:
	 	\begin{enumerate}
	 		\item $u_{n,0}=0$ and
	 		\item $u_{n,j+1}$ is a non-negative weak solution to the problem
	 		\begin{align}\label{approx sing help}
	 		\frac{\partial}{\partial t}u_{n,j+1}-\Delta_{p(x)}u_{n,j+1}&=\lambda h_n(u_{n,j}) +
	 		\frac{ g}{(u_{n,j+1}+\frac{1}{n})^{\delta(x)}}+ f_{n}&&\text{in}~Q_T,\nonumber\\
	 		u_{n,j+1}&= 0&&\text{on}~\Sigma_T,\\
	 		u_{n,j+1}(0,\cdot)&=u_{0,n}(\cdot)&&\text{in}~\Omega.\nonumber
	 		\end{align}
	 	\end{enumerate}
	 	\begin{lemma}\label{comp 1}
	 	Let $\lambda,\delta^+<1$ and let $f\in L^1(Q_T)$ verifies $f\geq v^{q(\cdot)-1}$ a.e. in $Q_T$, where $v$ is a weak solution to $\eqref{helping sing}$ obtained in Theorem $\ref{help sing exist}$. Then the sequence $\{u_{n,j}\}$ is increasing in $j$ and $u_{n,j}\leq v$ a.e. in $Q_T$ for every $j$.
	 	\end{lemma}
 	\begin{proof}
 	Clearly, by using the comparison principle Theorem $\ref{comparison}$, the sequence $\{u_{n,j}\}$ is increasing in $j$. From $\eqref{helping sing}$ and $\eqref{approx sing help}$ with $j=0$ we have
 		\begin{align}\label{comp}
 		\frac{\partial}{\partial t}u_{n,1}-\Delta_{p(x)}u_{n,1}-\frac{ g}{(u_{n,1}+\frac{1}{n})^{\delta(x)}}&\leq 2f\nonumber\\&=\frac{\partial }{\partial t}v-\Delta_{p(x)}v- \frac{g}{v^{\delta(x)}}\nonumber\\&\leq \frac{\partial }{\partial t}v-\Delta_{p(x)}v- \frac{ g}{(v+\frac{1}{n})^{\delta(x)}}.
 		\end{align}
  Let us define $\varphi_1=(u_{n,1}-T_k(v))^+$ with $k\gg\max\{\|u_{0,n}\|_{L^\infty(\Omega)},\|u_{n,1}\|_{L^{\infty}(Q_T)}\}$. This implies $\varphi_1\in V^{p(\cdot)}(Q_T)\cap L^\infty(Q_T)$ and $\varphi_1\equiv 0$ in the set $\{v>\|u_{n,1}\|_{L^{\infty}(\Omega)}\}$. Considering $\varphi_1 \chi_{\{0,t\}}$ for any $t\in (0,T]$ as a test function in $\eqref{comp}$ we get
  \begin{align}\label{comp special}
  \int_{Q_t}(u_{n,1}-T_k(v))_t&(u_{n,1}-T_k(v))^+\nonumber
  \\&+\int_{Q_t}\left(|\nabla u_{n,1}|^{p(x)-2}\nabla u_{n,1}-|\nabla T_k(v)|^{p(x)-2}\nabla T_k(v)\right)\cdot \nabla(u_{n,1}-T_k(v))^+ \leq 0.
  \end{align}
 Using the fact that $$\left(|\nabla u_{n,1}|^{p(x)-2}\nabla u_{n,1}-|\nabla T_k(v)|^{p(x)-2}\nabla T_k(v)\right)\cdot \nabla(u_{n,1}-T_k(v))^+\geq 0,$$
it follows from $\eqref{comp special}$ that
 \begin{align} 
 \frac{1}{2}\int_{Q_t}\frac{d}{dt}[(u_{n,1}-T_k(v))^+]^2\leq 0.\nonumber
 \end{align}
 		Since $u_{0,n}\leq T_k(v_0)$ a.e. in $\Omega$, we obtain $u_{n,1}\leq T_k(v)$ and hence $u_{n,1}\leq v$ a.e. in $Q_T$. \\
 		Let us assume $u_{n,j-1}\leq v$ a.e. in $Q_T$. Then, we have 
 \begin{align*}
 \frac{\partial}{\partial t}u_{n,j}-\Delta_{p(x)}u_{n,j}-\frac{ g}{(u_{n,j}+\frac{1}{n})^{\delta(x)}}&\leq \lambda h_n(u_{n,j-1})+ f_n\\&\leq u_{n,j-1}^{q(x)-1} +f_n\\&\leq v^{q(x)-1}+f\\&\leq 2f\nonumber\\&\leq \frac{\partial }{\partial t}v-\Delta_{p(x)}v- \frac{ g}{(v+\frac{1}{n})^{\delta(x)}}.
 \end{align*}
 		Thus, proceeding similarly we obtain that $u_{n,j}\leq v$ a.e. in $Q_T$ for every $j$.	
 	\end{proof}
 With the consideration of Lemma $\ref{comp 1}$, it is easy to prove that $u_{n,j}\uparrow u_n$ is a weak solution to problem $\eqref{approx}$ with $\beta=1$, $w_0=u_0$, and thus $u_n\leq v$ a.e. in $Q_T$. More precisely, $u_n$ satisfies:
  \begin{align}\label{approx sing}
 \frac{\partial }{\partial t}u_n-\Delta_{p(x)}u_n&=\lambda h_n(u_n)+{(u_n+1/n)^{-\delta(x)}}g+ f_n&&\text{in}~Q_T,\nonumber\\
 u_n&> 0&&\text{in}~Q_T,\nonumber\\
 u_n&= 0&&\text{on}~\Sigma_T,\\
 u_n(0,\cdot)&=u_{0,n}(\cdot)&&\text{in}~\Omega.\nonumber
 \end{align}
 According to Remark $\ref{remark L1}$, and the choice of $f$, we have $u_n^{q(\cdot)-1}\leq v^{q(\cdot)-1}\leq f\in L^{1}(Q_T)$ a.e. in $Q_T$. This proves the following result.
\begin{lemma}\label{lemma q(x)}
Let $\lambda,\delta^+< 1$ and let $f\in L^1(Q_T)$ verifies $f\geq v^{q(\cdot)-1}$ a.e. in $Q_T$. Further, let $u_n$ be a weak solution to the approximating  problem $\eqref{approx sing}$. Then, the sequence $\{u_n\}$ is uniformly bounded in $L^{q(\cdot)-1}(Q_T)$.
\end{lemma}
	\begin{proof}[Proof of Theorem $\ref{main result 1}$]
Let $u_n$ be a non-negative weak solution to $\eqref{approx sing}$. Then, from Lemma $\ref{lemma q(x)}$, the sequence $\{u_n\}$ is uniformly bounded in $L^{q(\cdot)-1}(Q_T)$ and $u_n^{q(\cdot)-1}\leq f$ a.e. in $Q_T$. Thus, using the generalized Lebesgue dominated convergence theorem, $u_n^{q(\cdot)-1}\rightarrow u^{q(\cdot)-1}$ strongly in $L^1(Q_T)$.\\
 By readapting the methods used in Lemma $\ref{infinity bound}$ and Lemma $\ref{lemma1 para}$, we prove $\{u_n\}$ to be bounded in $L^\infty\left(0,T;L^1(\Omega)\right)\cap V^{r(\cdot)}(Q_T)$, for every $1\leq r(\cdot)<p(\cdot)-\frac{N}{N+1}$. Moreover, for every $k>0$, $\{T_k(u_n)\}$ is bounded in $V^{p(\cdot)}(Q_T)$. Therefore, there exists a function $u$ such that, up to a subsequence, $u_n$ converges to $u$ a.e. in $Q_T$ and weakly in $V^{r(\cdot)}(Q_T)$. \\
Following the lines from step 1 and step 2 of Theorem $\ref{theorem 1}$, we establish that $u_n\rightarrow u$ strongly in $L^1\left(0, T;L^1_{loc}(\Omega)\right)$, $T_k(u_n)\rightarrow T_k(u)$ weakly in $V^{p(\cdot)}(Q_T)$, and 
	$$\lim\limits_{n\rightarrow\infty}\int_{Q_T}\frac{g}{(u_n+1/n)^{\delta(x)}}\varphi=\int_{Q_T}\frac{g}{u^{\delta(x)}}\varphi,~\forall \varphi\in C_c^1(\Omega\times[0,T)).$$ 
	Furthermore, replicating the proof of step 3 implies that $\{\nabla u_n\}$ converges almost everywhere to $\nabla u$ in $Q_T$, and $\nabla u_n\rightarrow \nabla u$  in $\left(L_{loc}^{{\tilde{p}}(\cdot)}(Q_T)\right)^N$ for any $\tilde{p}(\cdot)\leq p(\cdot)$. Indeed, for any $k>0,~n.m\in \mathbb{N}$, using the boundedness of the sequence $\left\{u_n^{q(\cdot)-1}\right\}$ in $L^1(Q_T)$ implies
\begin{equation}
\int_{Q_T}\left||u_n|^{q(x)-1}-|u_m|^{q(x)-1}\right|T_k(u_n-u_m)\phi^{p^+}\leq k \int_{Q_T}\left||u_n|^{q(x)-1}-|u_m|^{q(x)-1}\right|\phi^{p^+}\leq Ck.
\end{equation}
We are now in the position to pass the limit $n\rightarrow \infty$ in the weak formulation of $\eqref{approx sing}$, i.e in 
	\begin{align}
	-\int_{Q_T}u_{n}\varphi_t-\int_{\Omega}u_{0,n}\varphi(x,0)+\int_{Q_T}|\nabla u_{n}|^{p(x)-2}\nabla u_{n}\cdot \nabla\varphi&=\int_{Q_T}\lambda h_n(u_{n})\varphi+\int_{Q_T}\frac{g\varphi}{(u_{n}+\frac{1}{n})^{\delta(x)}}\nonumber\\&~~~~+\int_{Q_T} f_{n}\varphi\nonumber
	\end{align}
	for every $\varphi\in C_c^1(\Omega\times [0,T))$. Then, $u$ is a non-negative weak solution to $\eqref{main para}$, in the sense of Definition $\ref{weak main para}$, and $u$ satisfies
	\begin{align}
	-\int_{Q_T}u\varphi_t-\int_{\Omega}u_{0}\varphi(x,0)+&\int_{Q_T}|\nabla u|^{p(x)-2}\nabla u\cdot \nabla\varphi=\int_{Q_T}(\lambda u^{q(x)-1}+f)\varphi+\int_{Q_T}\frac{g\varphi}{u^{\delta(x)}}\nonumber
	\end{align} 
	for every $\varphi\in C_c^1(\Omega\times [0,T))$. This completes the proof.
	\end{proof}
	\section{Proof of Theorem $\ref{main result 2}$}\label{section 5}
	In this section, we discuss the problem $\eqref{main para 2}$, and prove our second main result, i.e. Theorem $\ref{main result 2}$. For this purpose, let us consider $u_n$ to be a weak solution to $\eqref{approx}$ with $\beta=0$ and $w_0=u_0$, i.e. $u_n$ satisfies:
	 	\begin{align}\label{approx main para 2}
	 	\frac{\partial u_n}{\partial t}-\Delta_{p(x)}u_n&=\lambda h_n(u_n) + (u_n+1/n)^{-\delta(x)}{g}&&\text{in}~Q_T,\nonumber\\
	 	u_n&> 0&&\text{in}~Q_T,\nonumber\\
	 	u_n&= 0&&\text{on}~\Sigma_T,\\
	 	u_n(0,\cdot)&=u_{0,n}(\cdot)&&\text{in}~\Omega.\nonumber
	 	\end{align}
	 	Here, $\lambda>0$, $u_{0,n}=T_n(u_0)$, $u_0\in L^r(\Omega)$ where $r$ satisfies $(B4)$, and the functions $p(\cdot),~q(\cdot)$ satisfy the hypotheses $(B1)-(B3)$.

\subsection{A priori estimates for $u_n$}\label{section 5.2}
To obtain a weak solution to the problem $\eqref{main para 2}$, we have to pass the limit $n\rightarrow\infty$ in the weak formulation of its approximating problem $\eqref{approx main para 2}$. Thus, it is essential to prove the following proposition which provides a uniform $L^\infty$ bound for $u_n$ independent of $n$.
\begin{proposition}\label{l infinity}
Let $u_n$ be a weak solution to $\eqref{approx main para 2}$. Then, there exists $\overline{T}>0$ such that for every $T<\overline{T}$, $u_n$ verifies the following properties:
\begin{enumerate}
	\item There exists $K_1>0$ such that $\|u_n(t)\|_{L^r(\Omega)}\leq K_1<\infty$ independent of $n$, for every $t\in (0,T)$.
	\item For every $\eta\in (0,T)$, there exists $K_\eta>0$ such that $\|u_n(t)\|_{L^\infty(\Omega)}\leq K_\eta<\infty$ independent of $n$, for every $t\in (\eta,T)$. 
\end{enumerate}	
\begin{proof}
	The proof follows the similar argument as in Proposition 4.2 of \cite{Bougherara}. For a fixed $\rho> \max\{\delta^--1,0\}$, let us multiply $u_n^{\rho+1}$ in $\eqref{approx main para 2}$ and then integrate by parts on $\Omega$ to obtain
	\begin{align}\label{first one}
	\frac{1}{2+\rho}\frac{d}{dt}\int_\Omega |u_n(t)|^{2+\rho}dx&+(1+\rho)\int_{\Omega}|\nabla u_n(t)|^{p(x)}|u_n(t)|^{\rho}\nonumber\\&=\lambda\int_{\Omega}h_n(u_n(t))u_n(t)^{1+\rho}+\int_{\Omega}{g}(u_n+1/n)^{-\delta(x)}u_n(t)^{1+\rho}\nonumber\\&\leq \lambda \int_{\Omega}u_n^{q(x)+\rho}+\int_{\Omega}{g}u_n^{1+\rho-\delta(x)}\nonumber\\&\leq \lambda\left(\int_{\Omega}u_n^{q^++\rho}+\int_{\Omega}u_n^{\rho}\right) +\int_{\Omega}{g}u_n^{1+\rho-\delta^-}+\int_{\Omega}{g}u_n^{1+\rho}\nonumber\\&\leq \lambda \int_{\Omega}u_n^{q^++\rho}+C_1\|u_n\|_{L^{2+\rho}(\Omega)}^{\rho}+C_2\|u_n\|_{L^{2+\rho}(\Omega)}^{1+\rho-\delta^-}+C_3\|u_n\|_{L^{2+\rho}(\Omega)}^{1+\rho}.
	\end{align} 
We have used the inequality $\theta^{\alpha}\leq \theta^\beta+1$ for any $\theta>0$ and $\alpha\leq \beta$ to obtain $\eqref{first one}$. It is known that $u_n\in V^{p(\cdot)}(Q_T)\cap L^\infty(Q_T)$. Thus, $u_n(t)\in W_0^{1,p^-}(\Omega)\cap L^\infty(\Omega)$ for almost every $t\in (0,T)$. From $(B3)$ we have $q^+<p^-(1+\frac{r}{N})$, where $r$ satisfies $(B4)$. Choose $\rho_0=r-2$. Thus, by applying Lemma 5.1 of \cite{Takac} for $p^-,q^+,r$, we establish the following Gagliardo - Nirenberg type estimate for every $\rho\geq \rho_0$.  
\begin{align}\label{G-N}
\lambda \int_{\Omega}u_n^{q^++\rho}&\leq \frac{1+\rho}{4}\int_{\Omega}|\nabla u_n|^{p^-}u_n^\rho+C_4\|u_n\|_{L^{2+\rho}(\Omega)}^{q^++\rho+(q^+-p^-)E(\rho)}\nonumber\\&\leq \frac{1+\rho}{4}\left(\int_{\Omega}|\nabla u_n|^{p(x)}u_n^\rho+\int_{\Omega}u_n^\rho\right)+C_4\|u_n\|_{L^{2+\rho}(\Omega)}^{q^++\rho+(q^+-p^-)E(\rho)}\nonumber\\&\leq \frac{1+\rho}{4}\int_{\Omega}|\nabla u_n|^{p(x)}u_n^\rho+C_5\|u_n\|_{L^{2+\rho}(\Omega)}^{\rho}+C_4\|u_n\|_{L^{2+\rho}(\Omega)}^{q^++\rho+(q^+-p^-)E(\rho)}
\end{align}
where $E(\rho)=\frac{q^+-2}{p^-\left(1+\frac{2+\rho}{N}\right)-q^+}$. Let us denote $B(\rho)=\frac{q^++\rho+(q^+-p^-)E(\rho)}{2+\rho}\geq 1$. On combining $\eqref{first one}$ and $\eqref{G-N}$ we have \begin{align}\label{same inq}
	\frac{1}{2+\rho}\frac{d}{dt}\|u_n(t)\|_{L^{2+\rho}(\Omega)}^{2+\rho}&\leq C_2\|u_n\|_{L^{2+\rho}(\Omega)}^{1+\rho-\delta^-}+C_3\|u_n\|_{L^{2+\rho}(\Omega)}^{1+\rho}+ C_6\|u_n\|_{L^{2+\rho}(\Omega)}^{\rho}+C_4\|u_n\|_{L^{2+\rho}(\Omega)}^{(2+\rho)B(\rho)}\nonumber\\&\leq C_7 \|u_n(t)\|_{L^{2+\rho}(\Omega)}^{(2+\rho)B(\rho)}+C_8.
\end{align}
The above inequality $\eqref{same inq}$ is a similar type of differential inequality as given in equation (4.13) of \cite{Bougherara}. Thus, using the same approach as used in the proof of Proposition 4.2 of \cite{Bougherara}, we can prove the first property (1).\\
It is left to prove (2). Let us consider a $C^1$ function $\rho:[0,T)\rightarrow[2,\infty)$ with $\rho^\prime(t)>0$ for every $t\in [0,T)$ and $\rho(0)=\rho_0=r-2$. Recalling the fact that $u_n\in L^{p^-}\left(0,T;W_0^{1,p(\cdot)}(\Omega)\right)\cap C([0,T];L^2(\Omega))\cap L^\infty(Q_T)$, and $ (u_n)_t\in L^{(p^-)'}\left(0,T;W_0^{-1,p'(\cdot)}(\Omega)\right)$, we deduce from Theorem 4.2 of \cite{Barbu} that $u_n\in W^{1,(p^-)'}\left(0,T;W_0^{-1,p'(\cdot)}(\Omega)\right)$. Thus, by using the result due to P. Tak\'{a}\v{c} [Lemma 4.1, \cite{Takac}] or Bougherara et al. [Lemma 4.1, \cite{Bougherara}], we obtain the following estimate:
\begin{align}\label{second one}
\frac{1}{2+\rho(t)}\frac{d}{dt}\| u_n(t)\|^{2+\rho(t)}_{L^{2+\rho(t)}(\Omega)}&=\int_{\Omega}\frac{\partial u_n}{\partial t}u_n^{1+\rho(t)}+\frac{\rho^\prime(t)}{2+\rho(t)}\int_\Omega u_n^{2+\rho(t)}\log u_n\nonumber\\
&=-(1+\rho(t))\int_{\Omega}|\nabla u_n(t)|^{p(x)}|u_n(t)|^{\rho}+\lambda\int_{\Omega}h_n(u_n(t))u_n(t)^{1+\rho}\nonumber\\&~~~~+\int_{\Omega}{g}(u_n+1/n)^{-\delta(x)}u_n(t)^{1+\rho}+\frac{\rho^\prime(t)}{2+\rho(t)}\int_\Omega u_n^{2+\rho}\log |u_n|.
\end{align}
According to Lemma 6.2 of \cite{Takac}, if $p^->2$ (the case $p^-=2$ is left to the reader), we have the following logarithmic estimate for every $0\leq t<T$.
\begin{align}\label{logrithmic}
\frac{\rho^\prime(t)}{2+\rho(t)}\int_\Omega u_n^{2+\rho}\log |u_n|&\leq \frac{1+\rho}{4}\int_{\Omega}|\nabla u_n|^{p^-}|u_n|^{\rho}+ C_9\left(\frac{\sigma \rho^\prime(t)}{(2+\rho)(1+\rho)^{2/p^-}}\right)^{\frac{p^-}{p^--2}}\|u_n\|_{L^{2+\rho}(\Omega)}^\rho\nonumber\\&~~~~+\psi_\sigma(\rho)\rho^\prime \|u_n\|_{L^{2+\rho}(\Omega)}^{2+\rho}+\frac{\rho^\prime(t)}{2+\rho(t)}\|u_n\|_{L^{2+\rho}(\Omega)}^{2+\rho} \log\|u_n\|_{L^{2+\rho}(\Omega)}\nonumber\\&\leq \frac{1+\rho}{4}\int_{\Omega}|\nabla u_n|^{p(x)}|u_n|^{\rho}+ C_{10}\|u\|_{L^{2+\rho}(\Omega)}^\rho\nonumber\\&~~~~+ C_9\left(\frac{\sigma \rho^\prime}{(2+\rho)(1+\rho)^{2/p^-}}\right)^{\frac{p^-}{p^--2}}\|u_n\|_{L^{2+\rho}(\Omega)}^\rho\nonumber\\&~~~~+\psi_\sigma(\rho)\rho^\prime \|u_n\|_{L^{2+\rho}(\Omega)}^{2+\rho}+\frac{\rho^\prime(t)}{2+\rho(t)}\|u_n\|_{L^{2+\rho}(\Omega)}^{2+\rho} \log\|u_n\|_{L^{2+\rho}(\Omega)}
\end{align} 
for $\sigma>0$ with $\psi_\sigma(\rho)=-\frac{N}{2(2+\rho)^2}\log\left(\frac{16\pi \rho}{2+\rho}\right)$. On using the estimates $\eqref{first one}-\eqref{same inq}$ and $\eqref{logrithmic}$ in $\eqref{second one}$ we establish
\begin{align}\label{refer}
\frac{1}{2+\rho(t)}\frac{d}{dt}\|u_n(t)\|^{2+\rho(t)}_{L^{2+\rho(t)}(\Omega)}dx & \leq C_7 \|u_n\|_{L^{2+\rho}(\Omega)}^{q^++\rho+(q^+-p^-)E(\rho)}+C_8+ C_{10}\|u\|_{L^{2+\rho}(\Omega)}^\rho\nonumber\\&~~~~+ C_9\left(\frac{\sigma \rho^\prime}{(2+\rho)(1+\rho)^{2/p^-}}\right)^{\frac{p^-}{p^--2}}\|u_n\|_{L^{2+\rho}(\Omega)}^\rho\nonumber\\&~~~~+\psi_n(\rho)\rho^\prime \|u_n\|_{L^{2+\rho}(\Omega)}^{2+\rho}+\frac{\rho^\prime(t)}{2+\rho(t)}\|u_n\|_{L^{2+\rho}(\Omega)}^{2+\rho} \log\|u_n\|_{L^{2+\rho}(\Omega)}\nonumber\\&\leq C_{11} \|u_n\|_{L^{2+\rho}(\Omega)}^{q^++\rho+(q^+-p^-)E(\rho)}+C_{12}+C_9\left(\frac{\sigma \rho^\prime}{(2+\rho)(1+\rho)^{2/p^-}}\right)^{\frac{p^-}{p^--2}}\|u_n\|_{L^{2+\rho}(\Omega)}^\rho\nonumber\\&~~~~+\psi_n(\rho)\rho^\prime \|u_n\|_{L^{2+\rho}(\Omega)}^{2+\rho}+\frac{\rho^\prime(t)}{2+\rho(t)}\|u_n\|_{L^{2+\rho}(\Omega)}^{2+\rho} \log\|u_n\|_{L^{2+\rho}(\Omega)}
\end{align}
for a.a. $t\in (0,T)$. With the consideration of $\eqref{refer}$, and the inequality (4.17) of Proposition 4.2 of \cite{Bougherara}, we conclude (2), i.e. for every $\eta\in (0,T)$, there exists $K_\eta>0$ such that $\|u_n\|_{L^\infty(\Omega)}\leq K_\eta<\infty$ independent of $n$, for every $t\in (\eta,T)$. 
\end{proof}
\end{proposition}
\noindent We now find the uniform Sobolev bounds for $u_n$.
	\begin{lemma}\label{lemma1-1 para}
		Let $\delta^+\leq1$. Then, the sequence $\{u_n\}$ is bounded in $V^{p(\cdot)}(Q_T)$.
		\begin{proof}
			We multiply $u_n$ in $\eqref{approx main para 2}$ and then integrate by parts on $Q_T$ to get
			\begin{align}
			\frac{1}{2}\left[\int_\Omega u_n^2(T)-\int_\Omega u_{0,n}^2\right]+\int_{Q_T}|\nabla u_n|^{p(x)}&=\int_{Q_T}\lambda h_n(u_n) u_n+\int_{Q_T}\frac{{g}u_n}{(u_n+1/n)^{\delta(x)}}\nonumber\\&\leq \int_{Q_T}\lambda u_n^{q(x)}+\int_{Q_T}{g}u_n^{1-\delta(x)}.\nonumber
			\end{align}
		 By Proposition $\ref{l infinity}$, the sequence $\{u_n\}$ is bounded in $L^\infty(0,T;L^r(\Omega))$, and $u_{0,n}\leq u_0\in L^r(\Omega)$. Thus, we get
			\begin{equation}
			\int_{Q_T}|\nabla u_n|^{p(x)}\leq C.\nonumber
			\end{equation}
			This implies $\{u_n\}$ is uniformly bounded in $V^{p(\cdot)}(Q_T)$.
		\end{proof}
	\end{lemma}
	\begin{lemma}\label{1-2 lemma}
		Let $\delta^+>1$. Then, $\{u_n\}$ is bounded in $V_{loc}^{p(\cdot)}(Q_T)$. Moreover, $\{u_n^{\frac{p^-+r-2}{p^-}}\}$ is bounded in $L^{p^-}(0,T;W_{0}^{1,p^-}(\Omega))$.
		\begin{proof}
			Let $\varphi\in \mathcal{D}(\Omega)$ be a non-negative function. We multiply  $(u_n-1)\varphi^{p^+}$ in $\eqref{approx main para 2}$ and then integrate by parts on $Q_T$ to get
			\begin{align}
			\frac{1}{2}\int_\Omega (u_n(T)-1)^2\varphi^{p^+}&-\frac{1}{2}\int_\Omega (u_{0,n}-1)^2\varphi^{p^+}+\int_{Q_T}|\nabla u_n|^{p(x)}\varphi^{p^+}\nonumber\\&~~~~+p^+\int_{Q_T}|\nabla u_n|^{p(x)-2}\nabla u_n\cdot\nabla\varphi \varphi^{p^+-1}(u_n-1)\nonumber\\&=\int_{Q_T}\lambda h_n(u_n) (u_n-1)\varphi^{p^+}+\int_{Q_T}\frac{{g}(u_n-1)\varphi^{p^+}}{(u_n+1/n)^{\delta(x)}}\nonumber\\&\leq \lambda C_\varphi\int_{Q_T}u_n^{q(x)} +\int_{\{u_n\leq 1\}}\frac{{g}(u_n-1)\varphi^{p^+}}{(u_n+1/n)^{\delta(x)}}+\int_{\{u_n> 1\}}\frac{{g}(u_n-1)\varphi^{p^+}}{(u_n+1/n)^{\delta(x)}}\nonumber\\&\leq \lambda C_\varphi\int_{Q_T}u_n^{q(x)} +\int_{\{u_n> 1\}\cap\{\delta(x)\leq 1\}}{g}u_n^{1-\delta(x)}\varphi^{p^+}+\int_{\{u_n> 1\}\cap\{\delta(x)> 1\}}{g}u_n^{1-\delta(x)}\varphi^{p^+}\nonumber\\&\leq C_\varphi\left[\lambda\int_{Q_T}u_n^{q(x)}+C_1\int_{\{u_n> 1\}\cap\{\delta(x)\leq 1\}}u_n^{1-\delta(x)}+C_2\right]\nonumber\\&\leq \tilde{C}_\varphi.
			\end{align}
We have used the boundedness of $\{u_n\}$ in $L^\infty(0,T;L^r(\Omega))$ to obtain the uniform bound $\tilde{C}_\varphi>0$ independent of $n$. Hence, by using Young's inequality we have 
\begin{align}
\int_{Q_T}|\nabla u_n|^{p(x)}\varphi^{p^+}&\leq \tilde{C}_\varphi+\frac{1}{2}\int_\Omega (u_{0,n}-1)^2\varphi^{p^+}+p^+\int_{Q_T}|\nabla u_n|^{p(x)-1}\cdot|\nabla\varphi| \varphi^{p^+-1}|u_n-1|\nonumber\\& \leq C_3+C_4p^+\epsilon \int_{Q_T}|\nabla u_n|^{p(x)} \varphi^{p^+}+p^+ C_\epsilon\int_{Q_T}|\nabla\varphi|^{p(x)}|u_n-1|^{p(x)}.\nonumber
\end{align}
This implies
\begin{equation}
\int_{Q_T}|\nabla u_n|^{p(x)}\varphi^{p^+}\leq C_5
\end{equation}
and hence $\{u_n\}$ is uniformly bounded is $V_{loc}^{p(\cdot)}(Q_T)$. With the consideration of the inequalities $\eqref{first one}-\eqref{same inq}$, with $2+\rho=r$, and by Proposition $\ref{l infinity}$, we have 
\begin{align}\label{r-2}
\frac{1}{r}\int_{0}^{T}\frac{d}{dt}\int_\Omega |u_n(t)|^{r}+\frac{3(r-1)}{4}\int_{0}^{T}\int_{\Omega}|\nabla u_n|^{p(x)}|u_n|^{r-2}&\leq  C_7 \int_{0}^{T}\|u_n\|_{L^{r}(\Omega)}^{rB(r-2)}+TC_8\nonumber\\& \leq \tilde{C}.
\end{align}
Since $p^-\leq p(x)$ for all $x\in \overline{\Omega}$, the above estimate $\eqref{r-2}$ gives
\begin{align}
\int_{Q_T}|\nabla u_n^{\frac{p^-+r-2}{p^-}}|^{p^-}&\leq C \int_{Q_T}|\nabla u_n|^{p^-}|u_n|^{r-2}\nonumber\\&\leq  C\int_{Q_T}|\nabla u_n|^{p(x)}|u_n|^{r-2}+C\int_{Q_T}|u_n|^{r-2}\nonumber\\&\leq C^*<\infty.\nonumber
\end{align}
This proves the lemma.
		\end{proof}
	\end{lemma}
\subsection{Proof of the main result}\label{section 5.1}
\begin{proof}[Proof of Theorem $\ref{main result 2}$]
Let $u_n$ be a weak solution to the problem $\eqref{approx main para 2}$. If $\delta^+\leq1$, then according to Proposition $\ref{l infinity}$ and Lemma $\ref{lemma1-1 para}$, there exists $\bar{T}>0$ such that the sequence $\{u_n\}$ is bounded in $V^{p(\cdot)}(Q_T)\cap L^\infty(0,T;L^r(\Omega))\cap L^\infty((\eta,T)\times \Omega)$ for every $T<\bar{T}$, every $\eta\in(0,T)$. Hence, there exists function $u$ such that, up to a subsequential level, as $n\rightarrow\infty$
\begin{equation}\label{para 1}
u_n\rightarrow u,~\text{weakly in}~V^{p(\cdot)}(Q_T),
\end{equation} and
\begin{equation}\label{para 2}
u_n\rightarrow u,~\text{weak star in}~L^\infty(0,T;L^r(\Omega))\cap L^\infty((\eta,T)\times \Omega).
\end{equation}
For the case $\delta^+> 1$, by considering Proposition $\ref{l infinity}$ and Lemma $\ref{1-2 lemma}$, there exists $\bar{T}>0$ and $u\in V_{loc}^{p(\cdot)}(Q_T)\cap L^\infty(0,T;L^r(\Omega))\cap L^\infty((\eta,T)\times \Omega)$ such that for every $T<\bar{T}$, $u_n$ converges to $u$ weakly to $V_{loc}^{p(\cdot)}(Q_T)$ and weak starly in $L^\infty(0,T;L^r(\Omega))\cap L^\infty((\eta,T)\times \Omega)$ for all $\eta\in(0,T)$.\\
The weak formulation of the problem $\eqref{approx main para 2}$ is given by 
\begin{equation}\label{weak approx main 2-2}
-\int_{Q_T}u_n\varphi_t-\int_{\Omega}u_{0,n}\varphi(x,0)+\int_{Q_T}|\nabla u_n|^{p(x)-2}\nabla u_n\cdot \nabla\varphi=\lambda\int_{Q_T} h_n(u_n)\varphi +\int_{Q_T}\frac{{g}\varphi}{(u_n+1/n)^{\delta(x)}}
\end{equation}
for every $\varphi\in C_c^1([0,T)\times\Omega)$. Following the proof of Theorem $\ref{theorem 1}$, we can pass the limit $n\rightarrow \infty$ in $\eqref{weak approx main 2-2}$ only if we are able show the a.e. convergence of $\nabla u_n$ towards $\nabla u$.\\
In order to tackle with the time derivative of $u$, we use the time regularization of $u$, refer Definition 5 \cite{Landes}. Let $u_r:Q_T\rightarrow\mathbb{R}$, for $r\in \mathbb{N}$, be the regularization of $u$ in time defined by
$$u_r(t,x)=\int_{-\infty}^{t}\bar{u}(s,x)\cdot r e^{r(s-t)}ds,$$
where $\bar{u}$ is the zero extension of $u$. Then, $u_r\rightarrow u$ strongly in $V^{p(\cdot)}(Q_T)$ as $r\rightarrow\infty$ and $\frac{\partial u_r}{\partial t}=r(u-u_r)$.\\
Let us consider $\phi\in C_c^1(Q_T)$ such that $0\leq \phi\leq 1$. Thus, by multiplying $(u_n-u_r)\phi^{p^+}$ in $\eqref{approx main para 2}$ and integrating by parts over $Q_T$, we have  
\begin{align}\label{a.e.}
-&\frac{1}{2}\int_{Q_T}(u_n-u_r)^2(\phi^{p^+})_t+\int_{Q_T}(u_r)_t(u_n-u_r)\phi^{p^+}\nonumber\\&~~~~+p^+\int_{Q_T}|\nabla u_n|^{p(x)-2}\nabla u_n\cdot \nabla\phi (u_n-u_r)\phi^{p^+-1}+\int_{Q_T}|\nabla u_n|^{p(x)-2}\nabla u_n\cdot \nabla(u_n-u_r)\phi^{p^+}\nonumber\\&\leq\lambda\int_{Q_T} u_n^{q(x)-1}|u_n-u_r|\phi^{p^+} +\int_{Q_T}\frac{{g}(u_n-u_r)\phi^{p^+}}{(u_n+1/n)^{\delta(x)}}.
\end{align}
For a fixed $\gamma>0$, using Proposition $\ref{l infinity}$, the right hand side of $\eqref{a.e.}$ is estimated as follows:
\begin{align}\label{a.e. second}
&\lambda\int_{Q_T} u_n^{q(x)-1}|u_n-u_r|\phi^{p^+} +\int_{Q_T}\frac{{g}(u_n-u_r)\phi^{p^+}}{(u_n+1/n)^{\delta(x)}}\nonumber\\&\leq \lambda K_{supp(\phi)}^{q^\pm-1}\int_{Q_T} |u_n-u_r|\phi^{p^+}+\int_{Q_T\cap\{u_n\leq \gamma\}}\frac{{g}(u_n-u_r)\phi^{p^+}}{(u_n+1/n)^{\delta(x)}}+\int_{Q_T\cap\{u_n> \gamma\}}\frac{{g}(u_n-u_r)\phi^{p^+}}{(u_n+1/n)^{\delta(x)}}\nonumber\\&\leq  \lambda K_{supp(\phi)}^{q^\pm-1}\int_{Q_T} |u_n-u_r|\phi^{p^+}+ K_{supp(\phi)} \int_{Q_T\cap\{u_n\leq \gamma\}}\frac{{g}\phi^{p^+}}{(u_n+1/n)^{\delta(x)}}+C\gamma^{-\delta^\pm}\int_{Q_T}|u_n-u_r|\phi^{p^+},
\end{align}
where $K_{supp(\phi)}>0$ depends on the support of $\phi$, $K_{supp(\phi)}^{q^\pm-1}=\max\{K_{supp(\phi)}^{q^+-1},K_{supp(\phi)}^{q^--1}\}$, and $\gamma^{-\delta^\pm}=\max\{\gamma^{-\delta^+},\gamma^{-\delta^-}\}$.  Using H\"{o}lder's inequality, Lemma $\ref{lemma1-1 para}$ and Lemma $\ref{1-2 lemma}$, the third integral in the left hand side of $\eqref{a.e.}$ can be estimated in the
following way:
\begin{align}\label{third}
\int_{Q_T}|\nabla u_n|^{p(x)-1}\phi^{p^+-1} |u_n-u_r||\nabla\phi|&\leq C_1\||\nabla u_n|^{p(\cdot)-1} \phi^{p^+-1}\|_{L^{\frac{p(\cdot)}{p(\cdot)-1}}(Q_T)}\||u_n-u_r|
|\nabla\phi|\|_{L^{p(\cdot)}(Q_T)}\nonumber\\&\leq C_2 \||u_n-u_r||\nabla\phi|\|_{L^{p(\cdot)}(Q_T)}.
\end{align}
 On combining $\eqref{a.e.}-\eqref{third}$ we obtain
\begin{align}\label{a.e. 1}
\int_{Q_T}&|\nabla u_n|^{p(x)-2}\nabla u_n\cdot \nabla(u_n-u_r)\phi^{p^+}\nonumber\\&\leq  \lambda K_{supp(\phi)}^{q^\pm-1}\int_{Q_T} |u_n-u_r|\phi^{p^+}+ K_{supp(\phi)} \int_{Q_T\cap\{u_n\leq \gamma\}}\frac{{g}\phi^{p^+}}{(u_n+1/n)^{\delta(x)}}+C\gamma^{-\delta^\pm}\int_{Q_T}|u_n-u_r|\phi^{p^+}\nonumber\\&~~~~+ C_2 \||u_n-u_r|\nabla\phi\|_{L^{p(\cdot)}(\Omega)}+\frac{1}{2}\int_{Q_T}(u_n-u_r)^2(\phi^{p^+})_t-\int_{Q_T}(u_r)_t(u_n-u_r)\phi^{p^+}.
\end{align}
We use the notation $w(n,r,\gamma)$ for a quanity which has the following characteristic: 
$$\lim\limits_{\gamma\rightarrow0}\lim\limits_{r\rightarrow\infty}\lim\limits_{n\rightarrow\infty}w(n,r,\gamma)=0.$$ 
By using the properties of $u_r$, $\eqref{para 1}$, and by readapting the steps used to proof $\eqref{singular term}$, we denote the right hand side of $\eqref{a.e. 1}$ as $w(n,r,\gamma)$. It follows by $\eqref{a.e. 1}$ adding and subtracting $\int_{Q_T}|\nabla u_r|^{p(x)-2}\nabla u_r\cdot \nabla(u_n-u_r)$, and then using the properties of $u_r$ we obtain
\begin{equation}\label{a.e. 2}
\int_{Q_T}\left(|\nabla u_n|^{p(x)-2}\nabla u_n-|\nabla u|^{p(x)-2}\nabla u\right)\cdot \nabla(u_n-u)\phi^{p^+}\leq \omega(n,r,\gamma).
\end{equation}
 We have used the boundedness of $\{|\nabla u_n|\}$ in $L_{loc}^{p(\cdot)-1}(Q_T)$ to obtain $\eqref{a.e. 2}$. On passing the limit first on $n$, then on $r$ (for $0<\gamma<1$ fixed) and finally on $\gamma$ in $\eqref{a.e. 2}$, we establish the following.
$$\lim\limits_{n\rightarrow\infty}\int_{Q_T}\left(|\nabla u_n|^{p(x)-2}\nabla u_n-|\nabla u|^{p(x)-2}\nabla u\right)\cdot \nabla(u_n-u)\phi^{p^+}\leq 0,$$
which implies $$\lim\limits_{n\rightarrow\infty}\int_{Q_T}|\nabla(u_n-u)|^{p(x)}\phi^{p^+}=0.$$
This proves that $\nabla u_n$  converges to $\nabla u$ a.e. in $Q_T$. Thus, we obtain a non-negative weak solution $u$ to $\eqref{main para 2}$, in the sense of Definition $\ref{weak main para 2}$. Moreover, $u$ satisfies the following equation.  
\begin{equation}\label{weak main para 2-2}
-\int_{Q_T}u\varphi_t-\int_{\Omega}u_{0}\varphi(x,0)+\int_{Q_T}|\nabla u|^{p(x)-2}\nabla u\cdot \nabla\varphi=\lambda\int_{Q_T} u^{q(x)-1}\varphi +\int_{Q_T}\frac{{g}\varphi}{u^{\delta(x)}},
\end{equation}
for every $\varphi\in C_c^1([0,T)\times\Omega)$. By considering Lemma $\ref{lemma1-1 para}$, Lemma $\ref{1-2 lemma}$, and by readapting the methods used in the step 5 of Theorem $\ref{help sing exist}$, we obtain $tr~u(\cdot,t)=0$ for almost every $t\in (0,T)$ in the sense of Definition $\ref{trace}$. Thus, we conclude the proof. 
\end{proof}	
\section*{Appendix}
	 In this section, we discuss a semi-discretization approach in time to study problems of type $\eqref{approx}$.  For this, we first deal with the corresponding stationary problems. Let $M\gg1$, $\eta=\frac{T}{M}$, and define $t_m=m\eta$ for $m\in \mathbb{N}$ with $1\leq m\leq M$. Consider the following elliptic problem:
	 \begin{align}\label{elliptic}
	 \frac{w_n^{(m)}-w_n^{(m-1)}}{\eta}-\Delta_{p(x)}w_n^{(m)}&=\lambda h_n\left(w_n^{(m-1)}\right) + \left(w_n^{(m)}+1/n\right)^{-\delta(x)}g+\beta[f_n]_\eta((m-1)\eta)&&\text{in}~\Omega,\nonumber\\
	 w_n^{(m)}&> 0&&\text{in}~\Omega,\\
	 w_n^{(m)}&= 0&&\text{on}~\partial\Omega,\nonumber
	 \end{align}
	 where $[f_n]_\eta$ is the Steklov average of $f_n$ given by $[f_n]_\eta(x,t)=\frac{1}{\eta}\int_{t}^{t+\eta}f_n(x,s)ds$.
	 The iteration starts from the initial condition $w_n^0=w_{0,n}\in L^\infty(\Omega)$, and $w_n^1$ satisfies 
	 \begin{align}
	 \frac{w_n^1-w_n^{0}}{\eta}-\Delta_{p(x)}w_n^1&=\lambda h_n(w_n^{0}) + (w_n^1+1/n)^{-\delta(x)}g+\beta[f_n]_\eta(0)&&\text{in}~\Omega,\nonumber\\
	 w_n^1&> 0&&\text{in}~\Omega,\\
	 w_n^1&= 0&&\text{on}~\partial\Omega.\nonumber
	 \end{align}
	 It is not difficult to prove the existence of a weak solution $w_n^{(m)}$ to $\eqref{elliptic}$ in $W_0^{1,p(\cdot)}(\Omega)$ for any $m\geq1$.\\
	  For $1\leq m\leq M$, and $t\in [t_{m-1},t_m)$, motivated by the implicit Euler's method, we define the functions $w_{n,\eta}$ and $\tilde{w}_{n,\eta}$ as follows.
	 $$w_{n,\eta}(\cdot,t)=w_n^{(m)}(\cdot)$$ and
	 $$\tilde{w}_{n,\eta}(\cdot,t)=\frac{w_n^{(m)}(\cdot)-w_n^{(m-1)}(\cdot)}{\eta}(t-t_{m-1})+w_n^{(m-1)}(\cdot),~~w_{n,\eta}(\cdot,0)=w_n^0(\cdot)=w_{0,n}(\cdot).$$ 
	 Clearly, $w_{n,\eta}$ and $\tilde{w}_{n,\eta}$ satisfy
	 \begin{align}\label{eta T}
	 \frac{\partial \tilde{w}_{n,\eta}}{\partial t}-\Delta_{p(x)}w_{n,\eta}=\lambda& h_n(w_{n,\eta}(\cdot-\eta)) + (w_{n,\eta}+1/n)^{-\delta(x)}g\nonumber\\&~~~~+\beta[f_n]_\eta(\cdot-\eta)~\text{in}~(\bar{\eta},T)~\text{for any}~\bar{\eta}>\eta.
	 \end{align}
	 We will now follow the proof of Theorem 3.1 of \cite{Bougherara} to obtain some uniform estimates for $w_{n,\eta}$ and $\tilde{w}_{n,\eta}$ independent of $\eta$.\\
	 \textbf{Claim 1:} The sequences $\{w_{n,\eta}\}$ and $\{\tilde{w}_{n,\eta}\}$ are uniformly bounded in $V^{p(\cdot)}(Q_T)$ and $V^{p(\cdot)}((\bar{\eta},T)\times\Omega)$, respectively, for every $0<\bar{\eta}<T$ independent of $\eta$.
	 \begin{proof}
	 	Let us multiply $\eta w_{n}^{(m)}$ in $\eqref{elliptic}$. Then, we integrate over $\Omega$	and take the sum from $m=1$ to $\bar{m}\leq M$ to obtain
	 	\begin{align}\label{energy 1 estimate}
	 	\sum_{m=1}^{\bar{m}}&\int_{\Omega}(w_n^{(m)}-w_n^{(m-1)})w_n^{(m)}+\eta  \sum_{m=1}^{\bar{m}}\int_{\Omega} |\nabla w_n^{(m)}|^{p(x)}\nonumber\\&=\eta  \sum_{m=1}^{\bar{m}}\int_{\Omega}\frac{gw_n^{(m)}}{(w_{n}^{(m)}+1/n)^{\delta(x)}}+\beta\eta\sum_{m=1}^{\bar{m}}\int_\Omega[f_n]_\eta((m-1)\eta)w_{n}^{(m)}+\lambda \eta  \sum_{m=1}^{\bar{m}}\int_{\Omega} h_n(w_n^{(m-1)})w_n^{(m)}.
	 	\end{align}
	 	 The first term in the left hand side of $\eqref{energy 1 estimate}$ can be rewritten as follows.
	 	\begin{align}
	 	\sum_{m=1}^{\bar{m}}\int_{\Omega}(w_n^{(m)}-w_n^{(m-1)})w_n^{(m)}&=\sum_{m=1}^{\bar{m}}\int_{\Omega}\frac{1}{2}\left[(w_n^{(m)}-w_n^{(m-1)})^2+(w_n^{(m)})^2-(w_n^{(m-1)})^2\right]\nonumber\\&=\frac{1}{2}\sum_{m=1}^{\bar{m}}\int_{\Omega}(w_n^{(m)}-w_n^{(m-1)})^2+\frac{1}{2} \int_{\Omega}(w_n^{\bar{m}})^2-\frac{1}{2}\int_{\Omega}(w_n^0)^2.\nonumber
	 	\end{align}
	 	By using the Young's inequality and by substituting the above estimates in $\eqref{energy 1 estimate}$ we get
	 	$$\frac{1}{2}\sum_{m=1}^{\bar{m}}\int_{\Omega}(w_n^{(m)}-w_n^{(m-1)})^2+\frac{1}{2} \int_{\Omega}(w_n^{\bar{m}})^2+\eta  \sum_{m=1}^{\bar{m}}\int_{\Omega} |\nabla w_n^{(m)}|^{p(x)}\leq (1+\lambda+\beta) \eta\sum_{m=1}^{\bar{m}}\int_{\Omega}(w_n^{(m)})^2+{C}(n).$$
	 	Thus, by \cite{Bougherara} it follows that $\{w_{n,\eta}\}$ and $\{\tilde{w}_{n,\eta}\}$ are bounded in $L^\infty(0,T;L^2(\Omega))$, independently of $\eta$. This proves the claim. 
	 \end{proof}
	 \noindent \textbf{Claim 2:} $\{w_{n,\eta}\}$ and $\{\tilde{w}_{n,\eta}\}$ are bounded in $L^\infty(Q_T)$.
	 \begin{proof}
	 	Let $\|f_n\|_{L^\infty(\Omega)}=M_n$, and let us consider a positive function $v_n$ that satisfies:
	 	\begin{align}
	 	\frac{\partial v_n}{\partial t}-\Delta_{p(x)}v_n&=\lambda\cdot n + n^{\delta(x)}g+\beta M_n&&\text{in}~Q_T,\nonumber\\
	 	v_n&= 0&&\text{on}~\Sigma_T,\\
	 	v_n&= n&&\text{in}~\Omega.\nonumber
	 	\end{align}
	 	Denote $v_n^{(m)}=v_n(t_m)$ for $0\leq m\leq M$. Thus, $v_n^{(m)}$ satisfies:
	 	\begin{align}
	 	\frac{v_n^{(m)}-v_n^{(m-1)}}{\eta}-\Delta_{p(x)}v_n^{(m)}&=\lambda \cdot n + n^{\delta(x)}g+\beta M_n&&\text{in}~\Omega,\nonumber\\
	 	v_n^{(m)}&> 0&&\text{in}~\Omega,\\
	 	v_n^{(m)}&= 0&&\text{on}~\partial\Omega.\nonumber
	 	\end{align}
	 	Therefore, using the weak comparison principle we guarantee that $v_n^{(m)}$ is a weak supersolution to $\eqref{elliptic}$, and for every $0\leq m\leq M$, $w_n^{(m)}\leq v_n^{(m)}\leq C(T)<\infty$ independently of $m$.
	 \end{proof}
	 \noindent\textbf{Claim 3:} $\{w_{n,\eta}\}$ is bounded in $L^\infty(\bar{\eta},T;W_0^{1,p(\cdot)}(\Omega))$, and $\left\{\frac{\partial \tilde{w}_{n,\eta}}{\partial t}\right\}$ is bounded in $L^2((\bar{\eta},T)\times \Omega)$ for every $0<\bar{\eta}<T$. Moreover, $\left\{\frac{\partial \tilde{w}_{n,\eta}}{\partial t}\right\}$ is bounded in $V^{p(\cdot)}(Q_T)^*$.
	 \begin{proof}
	 	Let us multiply $\frac{t_m+t_{m-1}}{2}(w_n^{(m)}-w_n^{(m-1)})$ in $\eqref{elliptic}$. Then, we integrate over $\Omega$	and take the sum from $m=2$ to $\bar{m}\leq M$ to obtain the following equation.
	 	\begin{align}\label{second energy estimate}
	 	\frac{1}{2}&\sum_{m=2}^{\bar{m}}(t_m+t_{m-1})\left(\eta\int_{\Omega}\left(\frac{w_n^{(m)}-w_n^{(m-1)}}{\eta}\right)^2 +\int_{\Omega}|\nabla w_n^{(m)}|^{p(x)-2}\nabla w_n^{(m)} \cdot \nabla(w_n^{(m)}-w_n^{(m-1)})\right)\nonumber\\&=\frac{1}{2}\sum_{m=2}^{\bar{m}}(t_m+t_{m-1})\left(\int_{\Omega}g(w_n^{(m)}+1/n)^{-\delta(x)} (w_n^{(m)}-w_n^{(m-1)})+\int_{\Omega}\lambda h_n(w_n^{(m-1)}) (w_n^{(m)}-w_n^{(m-1)})\right)\nonumber\\&~~+\frac{\beta}{2}\sum_{m=2}^{\bar{m}}(t_m+t_{m-1})\int_\Omega [f_n]_\eta((m-1)\eta)(w_n^{(m)}-w_n^{(m-1)}).
	 	\end{align}
	 	Using the properties of convex functions, we estimate the followings:
	 	\begin{align}
	 	&\frac{1}{2}\sum_{m=2}^{\bar{m}}(t_m+t_{m-1})\int_{\Omega}|\nabla w_n^{(m)}|^{p(x)-2}\nabla w_n^{(m)} \cdot (\nabla w_n^{(m)}-\nabla w_n^{(m-1)})\nonumber\\&\geq \sum_{m=2}^{\bar{m}} \frac{1}{2p^-}(t_m+t_{m-1})\int_{\Omega}(|\nabla w_n^{(m)}|^{p(x)}-|\nabla w_n^{(m-1)}|^{p(x)})\nonumber\\&=\frac{t_{\bar{m}}}{p^-}\int_\Omega|\nabla w_n^{(\bar{m})}|^{p(x)}-\frac{\eta}{p^-}\int_\Omega|\nabla w_n^{(1)}|^{p(x)}-\frac{\eta}{2p^-}\sum_{m=2}^{\bar{m}}\int_\Omega(|\nabla w_n^{(m)}|^{p(x)}+|\nabla w_n^{(m-1)}|^{p(x)})\nonumber\\&\geq \frac{t_{\bar{m}}}{p^-}\int_\Omega|\nabla w_n^{(\bar{m})}|^{p(x)}-\frac{2}{p^-}\int_{0}^{t_{\bar{m}}}\int_\Omega|\nabla w_{n,\eta}|^{p(x)},\nonumber
	 	\end{align}
	 	and 
	 	\begin{align}
	 	&\frac{1}{2}\sum_{m=2}^{\bar{m}}(t_m+t_{m-1})\int_{\Omega}g(w_n^{(m)}+1/n)^{-\delta(x)} (w_n^{(m)}-w_n^{(m-1)})\nonumber\\&\leq \frac{1}{2(1-\delta^+)}\sum_{m=2}^{\bar{m}}(t_m+t_{m-1})\int_{\Omega}g\left((w_n^{(m)}+1/n)^{1-\delta(x)}-(w_n^{(m-1)}+1/n)^{1-\delta(x)}\right)\nonumber\\&=\frac{t_{\bar{m}}}{2(1-\delta^+)}\int_{\Omega}g(w_n^{(\bar{m})}+1/n)^{1-\delta(x)}-\frac{\eta}{2(1-\delta^+)}\int_{\Omega}g(w_n^{(1)}+1/n)^{1-\delta(x)}\nonumber\\&~~~~-\frac{\eta}{(1-\delta^+)}\sum_{m=2}^{\bar{m}}\int_{\Omega}\left(g(w_n^{(m)}+1/n)^{1-\delta(x)}+(w_n^{(m-1)}+1/n)^{1-\delta(x)}\right)\nonumber\\&\leq \frac{t_{\bar{m}}}{2(1-\delta^+)}\int_{\Omega}g(w_n^{(\bar{m})}+1/n)^{1-\delta(x)} +C\int_{0}^{t_{\bar{m}}}\int_{\Omega}g(w_{n,\eta}+1/n)^{1-\delta(x)}.\nonumber
	 	\end{align}
	 	The last two terms of $\eqref{second energy estimate}$ can be approximated using Young's inequality.
	 	\begin{align}
	 	\frac{1}{2}&\sum_{m=2}^{\bar{m}}(t_m+t_{m-1})\int_{\Omega}\left[\lambda h_n(w_n^{(m-1)})+\beta[f_n]_\eta((m-1)\eta)\right] \left(w_n^{(m)}-w_n^{(m-1)}\right)\nonumber\\&\leq \eta\sum_{m=2}^{\bar{m}}(t_m+t_{m-1})\left[\int_\Omega[\lambda h_n(w_n^{(m-1)})+\beta[f_n]_\eta((m-1)\eta)]^2+\frac{1}{4}\int_\Omega\left(\frac{w_n^{(m)}-w_n^{(m-1)}}{\eta}\right)^2\right]\nonumber\\&\leq 2T\int_{0}^{t_{\bar{m}}}\int_\Omega \left[\lambda h_n(w_{n,\eta})+\beta [f_n]_\eta\right]^2+\frac{\eta}{4}\sum_{m=2}^{\bar{m}}(t_m+t_{m-1})\int_\Omega\left(\frac{w_n^{(m)}-w_n^{(m-1)}}{\eta}\right)^2.\nonumber
	 	\end{align}
	 	Thus, by substituting the above bounds in $\eqref{second energy estimate}$, we have
	 	\begin{align}
	 	\frac{1}{2}\int_{0}^{t_{\bar{m}}}\int_\Omega t\left|\frac{\partial\tilde{w}_{n,\eta}}{\partial t}\right|^2+\frac{t_{(\bar{m})}}{p^-}\int_\Omega|\nabla w_n^{\bar{m}}|^{p(x)}&\leq \frac{1}{2}\int_{\Omega}(w_n^{(1)}-w_n^{(0)})^2+\frac{t_{(\bar{m})}}{2(1-\delta^+)}\int_{\Omega}g(w_n^{\bar{m}}+1/n)^{1-\delta(x)} \nonumber\\&~~~~+C\int_{0}^{t_{\bar{m}}}\int_{\Omega}g(w_{n,\eta}+1/n)^{1-\delta(x)}\nonumber\\&~~~~+2T\int_{0}^{t_{\bar{m}}}\int_\Omega \left[\lambda h_n(w_{n,\eta})+\beta [f_n]_\eta\right]^2+\frac{2}{p^-}\int_{0}^{t_{\bar{m}}}\int_\Omega|\nabla w_{n,\eta}|^{p(x)}.\nonumber
	 	\end{align}
	 	On using the estimates obtained from Claim 1 and Claim 2, we obtain 
	 	\begin{equation}\label{L2}
	 	\left\|\frac{\partial\tilde{w}_{n,\eta}}{\partial t}\right\|_{L^2(\bar{\eta},T;L^2(\Omega))}\leq C_1. ~\text{for any}~ \bar{\eta}\in(0,T),
	 	\end{equation} 
	 	and 
	 	\begin{equation}
	 	\underset{t\in (0,T)}{\sup} t \int_\Omega|\nabla w_{n,\eta}|^{p(x)}\leq \underset{0\leq \bar{m}\leq M}{\max} t_{\bar{m}}\int_\Omega|\nabla w_n^{\bar{m}}|^{p(x)}\leq C_2
	 	\end{equation}
	 	where $C_1,C_2$ are two positive constants independent of $\eta$. By considering $\eqref{eta T}$ we deduce that $\left\{\frac{\partial \tilde{w}_{n,\eta}}{\partial t}\right\}$ is bounded in $V^{p(\cdot)}(Q_T)^*$.
	 \end{proof}
	 \noindent We have obtained all the required energy estimates. We now prove the existence result for $\eqref{approx}$, i.e. we prove Lemma $\ref{exist approx sing}$.
	 	\begin{proof}[Proof of Lemma $\ref{exist approx sing}$]
	 		With the consideration of the energy estimates obatained from Claim 1, 2 and 3, there exists $w_n$ and $\tilde{w}_n$ such that, up to a subsequence, as $\eta\rightarrow 0^+$ (i.e. as $m\rightarrow \infty$)
	 		$$w_{n,\eta}\rightarrow w_n ~\text{weakly in}~ V^{p(\cdot)}(Q_T),$$
	 		$$\tilde{w}_{n,\eta}\rightarrow \tilde{w}_n~\text{ weak starly in}~L^\infty(Q_T)\cap L^\infty(\bar{\eta},T,W_0^{1,p(\cdot)}(\Omega)),~\forall \bar{\eta}\in (0,T),$$
	 		$$\frac{\partial\tilde{w}_{n,\eta}}{\partial t}\rightarrow\frac{\partial \tilde{w}_{n}}{\partial t}~\text{ weakly in}~V^{p(\cdot)}(Q_T)^*\cap L^2((\bar{\eta},T)\times\Omega),~\forall \bar{\bar{\eta}}\in (0,T).$$
	 		For $M\gg 1$, there exists a unique $M^\prime$ such that $\bar{\eta}\in (t_{M^\prime},t_{M^\prime+1}]$ and
	 		\begin{equation}\label{unique u}
	 		\|w_{n,\eta}-\tilde{w}_{n,\eta}\|_{L^\infty(\bar{\eta},T;L^2(\Omega))}\leq 2 \underset{M^\prime\leq m\leq  M}{\max} \|w_n^{(m)}-w_n^{(m-1)}\|_{L^2(\Omega)}\rightarrow 0, ~\text{as}~\eta\rightarrow 0^+.
	 		\end{equation}
	 		This implies, $w_n=\tilde{w}_n$ in $Q_T$. Since $w_n\in L^{p^-}\left(0,T;W_0^{1,p(\cdot)}(\Omega)\right)$, and the equation $\eqref{eta T}$ implies $(w_n)_t\in L^{(p^-)'}\left(0,T;W_0^{-1,p'(\cdot)}(\Omega)\right)$, Theorem 4.2 of [\cite{Barbu}, page 167] proves that $w_n\in C(0,T;L^2(\Omega))$. Further, according to the Aubin-Lions-Simon Lemma (see \cite{Simon}), and by $\eqref{unique u}$, we obtain the following compactness results:
	 		\begin{equation}\label{continuous 1}
	 		w_{n,\eta}\underset{\eta\rightarrow 0^+}{\longrightarrow} w_n ~\text{in}~L^2(Q_T),
	 		\end{equation}	
	 		and
	 		\begin{equation}\label{continuous 2}
	 		\tilde{w}_{n,\eta}\underset{\eta\rightarrow 0^+}{\longrightarrow} w_n ~\text{in}~L^2((\bar{\eta},T)\times\Omega),~\forall\bar{\eta}\in (0,T). 
	 		\end{equation}
	 		The next claim is that $w_n$ is a weak solution to $\eqref{approx}$. For this purpose, let us multiply $(w_{n,\eta}-w_n)$ in $\eqref{eta T}$ and integrate over $Q_{T,\bar{\eta}}=(\bar{\eta},T)\times\Omega$. Thus, we obtain the following.
	 		\begin{align}\label{weak form elliptic}
	 		\int_{Q_{T,\bar{\eta}}}\frac{\partial\tilde{w}_{n,\eta}}{\partial t}(w_{n,\eta}-w_n)&+\int_{Q_{T,\bar{\eta}}}|\nabla w_{n,\eta}|^{p(x)-2}\nabla w_{n,\eta}\cdot (w_{n,\eta}-w_n)\nonumber\\&=\lambda\int_{Q_{T,
	 		\bar{\eta}}}h_n(w_{n,\eta}(\cdot-\eta))(w_{n,\eta}-w_n)+\int_{Q_{T,\bar{\eta}}}g(w_{n,\eta}+1/n)^{-\delta(x)}(w_{n,\eta}-w_n)\nonumber\\&~~~~+\beta\int_{Q_{T,
	 		\bar{\eta}}}[f_n]_\eta(\cdot-\eta))(w_{n,\eta}-w_n).
	 		\end{align}
	 		By using the dominated convergence Theorem, $\eqref{continuous 1}$, $\eqref{continuous 2}$, and finally by using convexity argument, we obtain
	 		\begin{equation}
	 		\frac{1}{2}\int_\Omega(\tilde{w}_{n,\eta}-w_n)^2(T)+\frac{1}{p^+}\left(\int_{Q_{T,\bar{\eta}}}|\nabla w_{n,\eta}|^{p(x)}-\int_{Q_{T,\bar{\eta}}}|\nabla w_{n}|^{p(x)}\right)\leq o_{\eta}(1)+\frac{1}{2}\int_\Omega(\tilde{w}_{n,\eta}-w_n)^2(\bar{\eta}).
	 		\end{equation}
	 		This implies
	 		$$\lim\limits_{\eta\rightarrow 0^+}\int_{Q_{T,\bar{\eta}}}|\nabla w_{n,\eta}|^{p(x)}\leq \int_{Q_{T,\bar{\eta}}}|\nabla w_{n}|^{p(x)}.$$
	 		By using the weak convergence of $w_{n,\eta}$ to $w_n$ in $V^{p(\cdot)}(Q_T)$, we deduce that 
	 		$$\underset{\eta\rightarrow0^+}{{\lim}}\int_{Q_{T,\bar{\eta}}}|\nabla w_{n,\eta}|^{p(x)}=\int_{Q_{T,\bar{\eta}}}|\nabla w_{n}|^{p(x)}.$$
	 		Thus, $$\nabla w_{n,\eta}\rightarrow \nabla w_n \text{ in } L^{p(\cdot)}((\bar{\eta},T)\times \Omega),~\forall\bar{\eta}\in(0,T) \text{ as }\eta\rightarrow 0^+.$$
	 		Consequently, there exists a sequence $(\eta)_m\rightarrow 0$ as $m\rightarrow\infty$, such that
	 		\begin{equation}
	 		\nabla w_{n,(\eta)_m}\rightarrow \nabla w_n \text{ a.e. in } Q_T \text{ as } m\rightarrow \infty.
	 		\end{equation}
	 		Thus, by these compactness results, it is easy to show that $w_n$ is a weak solution to $\eqref{approx}$.\\
	 		It is left to prove that $w_n(0)=w_{0,n}$. Let us multiply $\varphi\in \mathcal{D}(\Omega)$ in $\eqref{eta T}$ and integrate over $(t_1,t_2)\times \Omega$ for $0<t_1<t_2$. Thus, we get
	 		\begin{align}\label{t_1}
	 		\int_\Omega \tilde{w}_{n,\eta}(t_2)\varphi&-\int_\Omega \tilde{w}_{n,\eta}(t_1)\varphi+\int_{t_1}^{t_2}\int_\Omega|\nabla w_{n,\eta}|^{p(x)-2}\nabla w_{n,\eta}\cdot\nabla\varphi\nonumber\\&=\int_{t_1}^{t_2}\int_\Omega g (w_{n,\eta}+1/n)^{-\delta(x)}\varphi+\int_{t_1}^{t_2}\int_\Omega h_n(w_{n,\eta}(\cdot-\eta))\varphi+\int_{t_1}^{t_2}\int_\Omega [f_n]_\eta(\cdot-\eta))\varphi.
	 		\end{align}
	 		Using $\eqref{continuous 2}$ and passing the limit $t_1\rightarrow 0$ in $\eqref{t_1}$ we have
	 		\begin{align}\label{t_2}
	 		\int_\Omega \tilde{w}_{n,\eta}(t_2)\varphi&-\int_\Omega {w}_{0,n}\varphi+\int_{0}^{t_2}\int_\Omega|\nabla w_{n,\eta}|^{p(x)-2}\nabla w_{n,\eta}\cdot\nabla\varphi\nonumber\\&=\int_{0}^{t_2}\int_\Omega g(w_{n,\eta}+1/n)^{-\delta(x)}\varphi+\int_{0}^{t_2}\int_\Omega h_n(w_{n,\eta}(\cdot-\eta))\varphi+\int_{0}^{t_2}\int_\Omega [f_n]_\eta(\cdot-\eta))\varphi.
	 		\end{align}
	 		On the interval $(0,\eta)$, we have used $h_n(w_{n,\eta}(\cdot-\eta))= h_n(w_{0,n})$ and $[f_n]_\eta(\cdot-\eta)=[f_n]_\eta(0)$. By using the Lebesgue theorem and the energy estimates for $w_{n,\eta}$, $\tilde{w}_{n,\eta}$, we pass the limit $\eta\rightarrow 0$ and $t_2\rightarrow 0$ in $\eqref{t_2}$ to obtain the following.
	 		\begin{equation}
	 		\underset{t\rightarrow 0}{\lim}\int_\Omega {w}_{n}(t)\varphi=\int_\Omega {w}_{0,n}\varphi.
	 		\end{equation}
	 		Since, $w_n\in C(0,T;L^2(\Omega))$, we deduce that $w_n(0,x)=w_{0,n}(x)$ for every $x\in \Omega$. This concludes the proof. 
	 	\end{proof}
\section*{Acknowledgement}
The author Akasmika Panda thanks the financial assistantship received from the Ministry of Human Resource Development (M.H.R.D.), Govt. of India. The author D. Choudhuri thanks the Science and Engineering Research Board (SERB), India for
the research grant (MTR/2018/000525) to carry out the research. Both the authors also acknowledge the facilities received from the Department of mathematics, National Institute of Technology Rourkela.
		 
		\end{document}